\documentclass[11pt]{amsart}
\usepackage{amsmath,amsfonts,amssymb,amsxtra,amscd,enumerate,amsthm,dsfont}
\usepackage{color}
\usepackage{latexsym}
\usepackage{epsfig}
  \usepackage{fullpage}
\usepackage{bm}

\usepackage{url}

\newcommand\e{\epsilon}

\newcommand\les{\lesssim}
\newcommand\ges{\gtrsim}

\newcommand\R{\mathbb{R}}

\newcommand\C{\mathbb{C}}
\newcommand\Z{\mathbb{Z}}
\newcommand\N{\mathbb{N}}
\renewcommand\S{\mathbb{S}}

\newcommand{\calE}{\mathcal E}
\newcommand{\calT}{\mathcal T}
\newcommand{\calH}{\mathcal H}
\newcommand{\calR}{\mathcal R}
\newcommand{\calN}{\mathcal N}

\newcommand{\ls}{{\lesssim}}

\newcommand\la{\langle}
\newcommand\ra{\rangle}
\newtheorem{theo}{Theorem}
\numberwithin{theo}{section} 
\newtheorem{lemma}[theo]{Lemma}
\newtheorem{prop}[theo]{Proposition}
\newtheorem{corol}[theo]{Corollary}
\newtheorem{defin}[theo]{Definition}
\newtheorem{conjecture}[theo]{Conjecture}

\newtheorem{con}{Condition}

\numberwithin{equation}{section}

\newcommand{\on}{\operatorname}

\usepackage{mathtools}
\newcommand{\defeq}{\vcentcolon=}

\DeclarePairedDelimiter\abs{\lvert}{\rvert}
\DeclarePairedDelimiter\norm{\lVert}{\rVert}
\DeclarePairedDelimiter\inner{\langle}{\rangle}

\makeatletter
\let\oldabs\abs
\def\abs{\@ifstar{\oldabs}{\oldabs*}}
\let\oldnorm\norm
\def\norm{\@ifstar{\oldnorm}{\oldnorm*}}
\let\oldinner\inner
\def\inner{\@ifstar{\oldinner}{\oldinner*}}
\makeatother

\usepackage{enumitem}

\usepackage[mathlines]{lineno}
\definecolor{linenumbergray}{gray}{0.48}

\begin{document}
\title{The $(n-1)$-linear restriction estimate in $\R^n$ with degenerate curvature}

\author[I.~Bejenaru]{Ioan Bejenaru} \address[I.~Bejenaru]{Department
  of Mathematics, University of California, San Diego, La Jolla, CA
  92093-0112 USA} \email{ibejenaru@ucsd.edu}

\author[T.~Chan]{Tiklung Chan} \address[T.~Chan]{Department
  of Mathematics, University of California, San Diego, La Jolla, CA
  92093-0112 USA} \email{tic017@ucsd.edu}

\begin{abstract} We revisit the $(n-1)$-linear restriction estimate in $\R^n$ proved by the first author in \cite{Be-6} and refine the result in the case when the curvature assumptions contain small parameters. The standard $L^p$ estimate is replaced by a more complex norm/quasi-norm of type $\ell^p L^\frac{2}{n-2}$ which reflects the geometry of wave packet propagation and the degenerate curvature assumptions. The primary motivation for this work stems from a concrete application to the study of the quadratic nonlinear Schr\"odinger equation in $\R^{3+1}$, which is done in a companion paper \cite{BeBoCh}. 
\end{abstract}


\maketitle

%\linenumbers % comment out to remove line numbers

\section{Introduction}

For $n \geq 2$, let $U \subset \R^{n-1}$ be an open, bounded and connected  neighborhood of the origin and let $\Sigma: U \rightarrow \R^{n}$ be a smooth parametrization of an $n-1$-dimensional submanifold of $\R^{n}$ (hypersurface), which we denote by $S=\Sigma(U)$. By a smooth parametrization we mean that $\Sigma$ satisfies
\begin{equation} \label{smooth}
\| \partial^\alpha \Sigma \|_{L^\infty(U)} \lesssim_\alpha 1,
\end{equation} 
for all multi-indices $\alpha$ with $|\alpha| \leq N$ for some large $N$. We then say that $S$ is a smooth hypersurface if it admits a parametrization satisfying \eqref{smooth}.
To such a parametrization of $S$ we associate the extension operator $\calE$  defined by
\[
\calE f(x) = \int_U e^{i x \cdot \Sigma(\xi)} f(\xi) d\xi. 
\]
where $f \in L^1(U)$ and $x \in \R^n$. This operator is closely related to a more intrinsic formulation of the extension operator
\begin{equation} \label{Eint}
\tilde \calE g(x)= \int_S e^{ix \cdot \omega} g(\omega) d \sigma_S(\omega),
\end{equation}
where $g \in L^1(S; d \sigma_S)$ and $x \in \R^n$. Indeed, using the parametrization $\Sigma$ we obtain
$\omega=\Sigma(\xi)$ and $d \sigma_S(\omega)=|\frac{\partial \Sigma}{\partial_{\xi_1}} \wedge ... \wedge \frac{\partial \Sigma}{\partial_{\xi_{n-1}}}| d \xi$. Thus,
with $f(\xi)= g(\Sigma(\xi)) |\frac{\partial \Sigma}{\partial_{\xi_1}} \wedge ... \wedge \frac{\partial \Sigma}{\partial_{\xi_{n-1}}}|$, the two formulations are equivalent. 
We will use good parameterizations in the sense that $ |\frac{\partial \Sigma}{\partial_{\xi_1}} \wedge ... \wedge \frac{\partial \Sigma}{\partial_{\xi_{n-1}}}| \approx 1$
throughout the domain and, given that all results are in terms of $L^p$ norms, the equivalence is carried out at the level of the estimates as well. 

Let $1\leq k\leq n$ and for each $i\in\{1,...,k\}$ let $S_i\subset\R^n$ be a smooth compact hypersurface with extension operator $\calE_i$. The $k$-linear restriction estimate is the inequality 
\begin{equation}  \label{MRE}
\| \prod_{i=1}^k \calE_i f_i \|_{L^p(\R^{n})} \les \prod_{i=1}^k \| f_i \|_{L^2(U_i)}.  
\end{equation}
In a more compact format this estimate is abbreviated as
\begin{equation} \label{Rsp}
\calR^*(2 \times ... \times 2 \rightarrow p).
\end{equation}
A natural condition to impose on the hypersurfaces is the standard transversality condition: there exists $\nu >0$ such that
\begin{equation} \label{trans}
| N_1(\zeta_1) \wedge ... \wedge N_{k}(\zeta_k) | \geq \nu,
\end{equation}
for any $\zeta_i \in S_i$, $1\leq i\leq k$; here $N_i(\zeta_i)$ is the unit normal at $\zeta_i$ to $S_i$ and the choice of orientation is not important. 

The second condition we impose is related to curvature: there exists  $\nu_1 >0$ such that for any $\zeta_i \in S_i, i \in \{1,...,k\}$, 
and for any $l \in \{1,...,k\}$ we have
\begin{equation} \label{curva}
| N_1(\zeta_1) \wedge ... \wedge N_{k}(\zeta_k) \wedge S_{N(\zeta_l)} v | \geq \nu_1 | v|,
\end{equation}
for every $v \in T_{\zeta_l} S_l$ with the property $v \perp N_i(\zeta_i), \forall i=1,...,k$. Here $S_{N(\zeta_l)}$ stands for the shape operator associated to the surface $S_l$ at $\xi_l$ - a formal definition is provided in Chapter \ref{tdg}. Based on simple continuity arguments, it can be easily shown that there is some flexibility in the above estimate in the sense that it holds true for any $v \in T_\zeta S_l$ such that $|v \wedge N_1(\zeta_1) \wedge ... \wedge N_{k}(\zeta_k)| \geq (1-\nu_2)  |v| |N_1(\zeta_1) \wedge ... \wedge N_{k}(\zeta_k)|$ for some $0 < \nu_2 \ll 1$, but this will not play a role in our analysis.

When accounting for dimensionality, \eqref{curva} implies that each surface has a non-zero principal curvature with a quantitative bound from below (for its absolute value) and the corresponding principal direction is a ``good" direction, i.e. a direction that is transversal to the normals to the other surfaces. In addition, the shape operator does not ``flip" that particular direction into a ``bad" direction.  

There are two main conjectures in the field concerning $k$-linear restriction estimates. We summarize them here and refer the reader to \cite{Be-6} for a more comprehensive presentation. The first one concerns the generic case, where no curvature assumption is made.

\begin{conjecture} \label{conj}
If the surfaces $S_i$ satisfy \eqref{trans} then $\calR^*(2 \times ... \times 2 \rightarrow p)$ holds true for any $p \geq p(k)=\frac{2}{k-1}$. 
\end{conjecture}
There has been substantial progress on this conjecture:  Bennett, Carbery and Tao established the almost optimal estimate at the endpoint in \cite{BeCaTa} and, more recently, Tao established the sharp estimate away from the endpoint in \cite{Tao-new}. The sharp endpoint estimate for Conjecture \ref{conj} is currently open. For the closely related multilinear Kakeya inequality, Guth established the sharp endpoint estimate in \cite{Gu-main}. 

The second conjecture concerns the non-generic case, where curvature assumptions are made in addition to the transversality condition.

\begin{conjecture} \label{conjj}
If the surfaces $S_i$ satisfy \eqref{trans} and \eqref{curva} then $\calR^*(2 \times ... \times 2 \rightarrow p)$ holds true for any $p \geq p(k)=\frac{2(n+k)}{k(n+k-2)}$. 
\end{conjecture}
 The case $k=1$ is the classical Stein-Tomas estimate, see for instance \cite{St}. The case $k=n$ brings no improvement on the exponent $p$ over the generic case as they both have the same optimal exponent of $\frac{2}{n-1}$. The bilinear case has been intensely studied, see for example \cite{Bou-CM, Wo, Tao-BW, Tao-BP, TV-CM1, Lee-BR,LeeVa, Be2, Va, Le, Te, Sto, BuMuVa, BaLeLe, TaVaVe} and the references therein. 

The remaining cases are $3 \leq k \leq n-1$. Here, the $k$-broad version of this conjecture was established by Guth using polynomial partitioning in \cite{Gu-II}. Other works using polynomial partitioning to prove new broad and linear restriction estimates include \cite{Gu-I, HiRo, Wa, GuWaZh} and references therein. In addition, Kakeya versions of Conjecture \ref{conjj} have been established by Guth and Zahl \cite{GuZa}, Zahl \cite{Za}, Hickman, Rogers and Zhang \cite{HiRoZh}.

As for the genuine multilinear restriction estimate, as formulated by Conjecture \ref{conjj}, the first author established it in the particular case when the surfaces have some conical structure, see \cite{Be3,Be4}. 
To the best of our knowledge, $k=n-1$ is the only known case (in the range $3 \leq k \leq n-1$)  for the genuine multilinear restriction estimate in the general setup, which was established by the first author in \cite{Be-6}.
\begin{theo}[{\cite[Theorem 1.2]{Be-6}}] \label{mainT}  Let $k=n-1$. Assume that $S_1,...,S_k$ satisfy \eqref{trans} and \eqref{curva}. 
Given any $p$ with $p(k)=\frac{2(n+k)}{k(n+k-2)}< p \leq \infty$, the following holds true
\begin{equation} \label{mainE}
\| \prod_{i=1}^k \calE_i f_i \|_{L^p(\R^{n+1})} \leq C(p) \prod_{i=1}^k \| f_i \|_{L^2(U_i)}, \quad \forall f_i \in L^2(U_i).   
\end{equation}
\end{theo}
We remark that one of the key novelties of this result is the clarification of the precise role that curvature plays in this problem, in particular the fact that Theorem \ref{mainT} applies to all surfaces satisfying the curvature and transversality conditions. Most of the previous works mentioned above focus on explicit surfaces such as paraboloids or cones, in which case \eqref{curva} is automatically satisfied under reasonable hypotheses. For example, in the case of the paraboloid, \eqref{curva} is implied by \eqref{trans}. While in the context of  multilinear restriction theory this geometric flexibility can be seen as a nice theoretical clarification of the role played by transversality and curvature, it turns out that it plays a crucial role in applications to the study of nonlinear dispersive PDEs. Unexpectedly, it turns out that even the very general setup of this theorem does not suffice for the particular application we have in mind.

In this paper we investigate a variation of this result with further relaxed geometric assumptions on the underlying hypersurfaces, which is motivated by a concrete application to the study of the quadratic nonlinear Schrödinger equation in $\R^{3+1}$, done in an upcoming companion paper \cite{BeBoCh}.

We keep the transversality condition \eqref{trans} as is, but we modify the curvature condition \eqref{curva}. In order to avoid unnecessary technical distractions, we assume that that the given hypersurfaces can be parametrized as graphs $S_i=\{(\xi,\tau): \xi \in D_i\subset \R^{n-1}, \tau \in \R\}\subset\R^n$ where $\tau=\varphi_i(\xi)$ with smooth $\varphi_i$. Then we assume that there exists $0 < \nu_i \leq 1$ such that 
\begin{equation} \label{sm23}
\|\partial^\alpha \varphi_i\|_{L^\infty(D_i)} \les \nu_i, \quad \forall |\alpha|=2,3. 
\end{equation}
As a consequence this gives the following estimate for the shape operator
\begin{equation} \label{Scurv}
| S_{N(\zeta_i)} v | \leq \nu_i |v|, \quad \forall v \in T_{\zeta_i} S_i, \quad \forall \zeta_i \in S_i,
\end{equation}
which implies that the principal eigenvalues of $S_{N(\zeta_i)}$ are bounded from above by $\nu_i$. 

Then we impose the corresponding variant of \eqref{curva}; there exists $c >0$ such that
\begin{equation} \label{curvad}
| N_1(\zeta_1) \wedge ... \wedge N_{k}(\zeta_k) \wedge S_{N(\zeta_i)} v | \geq c \nu_i | v|, 
\end{equation}
for every $v \in T_{\zeta_l} S_l$ with the property $v \perp N_i(\zeta_i), \forall i=1,...,k$. Practically, we relax \eqref{curva} by allowing each surface $S_i$ to have its own curvature parameter $\nu_i$ and do not require a universal bound from below for these parameters. We remark that as a consequence of \eqref{Scurv} and \eqref{curvad} we obtain the following
\begin{equation} \label{curvad2}
| S_{N(\zeta_i)} v | \geq c \nu_i |v|, \qquad |N_1(\zeta_1) \wedge ... \wedge N_{k}(\zeta_k) \wedge \frac{S_{N(\zeta_i)} v}{|S_{N(\zeta_i)} v|} | \geq c | v|
\end{equation}
for every $v \in T_{\zeta_l} S_l$ with the property $v \perp N_i(\zeta_i), \forall i=1,...,k$. In other words, these vectors $v$ have the full possible curvature $| S_{N(\zeta_i)} v | \approx \nu_i |v|$ (in light of \eqref{Scurv}), 
and $S_{N(\zeta_i)} v$ keeps the original transversality properties on the set of vectors $N_i (\zeta_i)$ which $v$ originally had. 

Our main goal is to investigate the behavior of the estimate \eqref{mainE} in this new context. One of the key observations is that the estimate cannot improve upon the generic $L^\frac2{k-1}$ multilinear estimate up to a basic scale: roughly speaking, that of the parallelepiped of size $\nu_1^{-1} \times ... \times \nu_k^{-1}$, where the parallelepiped is constructed in the directions $N_1 \times ... \times N_{k}$, with $N_i$ being a normal to $S_i, i\in\{1,...,k\}$. To be more precise, one could derive $L^p$ estimates with $p <\frac2{k-1}$ but this would lose factors of type $\prod_{i=1}^{k} \nu_i^{-\beta}$, which we are trying to avoid.

Without restricting the generality of the argument we can assume $N_i=e_i$, meaning that is there a choice of $\zeta_i \in S_i$ with $N_{\zeta_i} S_i =e_i, i = 1,...,k$. We let $\mathcal{P}$ be the set of space-time parallelepipeds with sides parallel to the coordinate axes and whose sides have length $\nu_i^{-1}$ in the direction of $e_i, 1\leq i\leq k$ respectively, and $1$ in the direction of $e_i$, $i=k+1,...,n$ - by convention, we let $\nu_i=1$ for $k+1\leq i\leq n$. Explicitly, $\mathcal{P}$ is defined by
\begin{align}\label{Pdef}
    \mathcal{P}=\{ \prod_{i=1}^k [\nu_i^{-1}l_i, \nu_i^{-1}(l_i+1)] ; l=(l_1,...,l_n)\in \Z^n\}.
\end{align}
Using this notation, we introduce the norm or quasi-norm
\[
\|F\|_{\ell^p_{\mathcal{P}} L^q}= \left( \sum_{P \in \mathcal{P}} \|F\|^p_{L^q(P)} \right)^\frac1p,
\]
for any $0<p,q \leq \infty$. For a measurable set $A \subset \R^n$, we use the notation $\ell^p_{\mathcal{P}} L^q(A)$ to denote the norm or quasi-norm above restricted to the set $A$. 

The main result of this paper is the following.
\begin{theo} \label{mainTd} Let $k=n-1$ and assume $S_i$ are as above. There exists $C(\bm{\nu})=C(\nu_1,...,\nu_{k})$ such that for any $p$ with $\frac{2(n+k)}{k(n+k-2)}< p \leq \infty$, the following holds true
\begin{align}\label{mainme}
    \| \prod_{i=1}^k  \mathcal{E}_i f_i\|_{\ell^p_\mathcal{P} L^{\frac{2}{k-1}}} \leq C(\bm{\nu}) \prod_{i=1}^k \|f_i\|_{L^2}
\end{align}

\end{theo}
The estimate \eqref{mainme} states that within the parallelepipeds $P$ one has access to the (optimal) $L^\frac{2}{k-1}$ generic multilinear estimate and gains are recorded as we perform the summation across the parallelepipeds $P \in \mathcal{P}$ - this is where the curvature effects come into play. Indeed, since $\frac{2(n+k)}{k(n+k-2)} < \frac2{k-1}$, the value of $p$ is allowed to go below $\frac2{k-1}$ and it is within this range that the estimate is stronger than the generic one in $\ell^\frac2{k-1} L^\frac2{k-1}= L^\frac2{k-1}(\R^n)$. The estimate is closely connected to the geometry of the parallelepipeds $P$ which were constructed as being the ``maximal" geometric structures within which the waves do not disperse; it is maximal in the sense that, for waves originating in $P$, dispersion effects only manifest outside $P$. Within a parallelepiped $P$ the wave $\calE_i f_i$ is essentially flat in the $x_i$ direction; in other words no curvature effects are recorded.

The constant $C(\bm{\nu})$ will be discussed in more detail in the following subsection but for now the reader may assume that $C(\bm{\nu}) \les_\e \max(\nu_1^{-1},...,\nu_k^{-1})^\epsilon$ for all $\e>0$. The bound on $C(\bm{\nu})$ comes solely from the generic estimate in Theorem \ref{MB} and it is not a feature of the intrinsic arguments used to establish this theorem. We believe that the estimate holds with $C(\bm{\nu})\lesssim 1$, but this would require a revision of the theory in \cite{Be5} about the generic multilinear restriction estimate which we do not pursue here.

Stating Theorem \ref{mainTd} in terms of the $\ell^p_\mathcal{P} L^{\frac{2}{k-1}}$ quasi-norm versus the standard $L^p$ quasi-norm is one of the main novelties in this paper. On one hand our analysis reveals a natural way in which the classical $L^p$ estimate in Theorem \ref{mainE} can be stated in this more general setup. On the other hand, stating the estimate in terms of the $\ell^p_\mathcal{P} L^{\frac{2}{k-1}}$ quasi-norm (versus the standard $L^p$) is essential in avoiding losses in $\bm{\nu}$ and this is important in applications. We remark that estimating the multilinear restriction estimate in this fashion is partially inspired by the recent work of the first author in \cite{Be-rs}, where the classical bilinear restriction estimate in $L^p$ with $p<2$ is restated in terms of the analogous $\ell^p L^2$ norm.

As mentioned previously, the primary motivation for this new set-up is its applicability to the analysis of nonlinear dispersive PDEs. In the companion paper \cite{BeBoCh}, we use the above result in the context of paraboloids to prove global well-posedness and scattering for the quadratic NLS in three spatial dimensions with small initial data in the critical Sobolev space. This specific equation was chosen as the simplest nonlinear model where the $(n-1)$-linear restriction estimate can be deployed, after developing the necessary transference theory of \eqref{mainme} to more general objects - this is crucial since free waves (such as $\calE_i f_i $) solve linear equations, not nonlinear ones. Ultimately, what is actually important in the context of PDEs is obtaining an $L^1$ global estimate, and the improved $\ell^p$ structure of \eqref{mainme} provides the room needed to obtain the desired $\ell^1_{\mathcal{P}} L^1=L^1(\R^n)$ estimate via this new transference theory. 

In the course of this work, we realized that we need to modify \eqref{curva} to accommodate degenerate curvatures. In particular the specific estimate we use is a trilinear estimate in $\R^4=\R^{3+1}$ for extension operators associated to the paraboloid, and the degenerate curvature condition appears once we study very high frequency pieces of the paraboloid. The precise statement needed in \cite{BeBoCh} is the following. 

\begin{corol}\label{cor qnls}
Let $p>\frac{14}{15}$. Let $\mathbb{P}=\{(\xi,|\xi|^2): \xi \in \R^3\}$ denote the paraboloid in $\R^4$ and let $S_i=\mathbb{P}\cap B_1(x_i)$ where $\abs{x_1},\abs{x_2}\leq 8$ and $\abs{x_3}\approx M$ such that they satisfy \eqref{trans} with $\nu\approx 1$, that is
\begin{align*}
    \abs{N(\zeta_1)\wedge N(\zeta_2) \wedge N(\zeta_3)} \gtrsim 1\quad\forall \zeta_i\in S_i.
\end{align*}
Let $\tilde{\mathcal{E}}_i$ denote the (intrinsic) extension operator associated to each $S_i$, as detailed in \eqref{Eint}. Then for all $\e>0$ there exists $C_\e>0$ such that
\begin{align}\label{eq cor estimate}
    \| \prod_{i=1}^3  \tilde{\mathcal{E}}_i f_i\|_{\ell^p_\mathcal{P} L^{1}} \leq C_\e \la M \ra ^{\e} \prod_{i=1}^3 \|f_i\|_{L^2}.
\end{align}
\end{corol}

We note that in the companion paper \cite{BeBoCh}, this result is used in the regime $M\gtrsim 1$. With this in mind, we choose to state the estimate using the intrinsic operators $\tilde{\mathcal{E}}$ in order to avoid extra details regarding the choice of parametrization in this case. In particular, the intrinsic extension operators for $S_1$ and $S_2$ are essentially equivalent to the extrinsic extension operator for the paraboloid and when $M \approx 1$, this is true for $S_3$ too. On the other hand, if $M\gg 1$, the intrinsic extension operator for $S_3$ is essentially equivalent to the extrinsic extension operator for the reparametrized paraboloid in lateral coordinates
\begin{align*}
    \mathcal{E}f(x,t)=\int e^{i((x_1,x_2,t)\cdot (\xi_1,\xi_2,\tau)+x_3\sqrt{\tau-\xi_1^2-\xi_2^2})} f (\xi_1,\xi_2,\tau) d\xi_1 d\xi_2 d\tau,
\end{align*}
where we assume here for convenience that $N_3(\zeta_3)\approx e_3$ for $\zeta_3\in S_3$. Generally speaking, in nonlinear PDE similar geometric formulations often appear when studying frequency-localized interactions. It may be instructive for the reader to keep the set-up from this corollary in mind as a representative example of the degenerate curvature condition \eqref{Scurv}. 

We postpone the proof of this corollary to the end of Section \ref{tdg}, after all the geometric objects are properly defined.

\subsection{Supporting results}
The proof of Theorem \ref{mainTd} involves polynomial partitioning arguments, wave packet theory and a fine geometric analysis of the role of played by the shape operator. In addition, just as in \cite{Be-6}, the proof of Theorem \ref{mainTd} uses two supporting results: a localized version of the generic multilinear restriction estimate and a multilinear $\e$-removal result. 

First, we discuss the localized multilinear restriction estimate. Theorem \ref{mainTd} requires some curvature information on the hypersurfaces involved. However, in its proof, a crucial role is played by the generic multilinear restriction estimate where only transversality assumptions are made. In particular, we use a refined version of these estimates with additional localization hypotheses. Such results were investigated in a very general setup by the first author in \cite{Be1,Be5}, and by Beltran, Duncan, and Hickman in \cite{BeDuHi}.
 
Given $S$ a submanifold of $\R^n$, we let $B_\epsilon(S)=\cup_{\zeta \in S} B_\epsilon(\zeta)$ and denote by $N_\zeta S$ the normal plane to $S$ at $\zeta$. Note that $N_i(\zeta_i)$ and $N_{\zeta_i} S_i$ use similar but different notation - the first is the (unit) normal vector to $S_i$ at $\zeta_i$, while the other one is the normal plane to $S_i$ at $\zeta_i$. These two essentially coincide when $S_i$ is a hypersurface but differ when the co-dimension of $S_i$ increases.

 If $V_i$ is a $d_i$-dimensional plane for any $1 \leq i \leq k$, then by $|V_1 \wedge ... \wedge V_k|$ we mean the quantity
$|v_{1,1} \wedge ... v_{1,d_1} \wedge ... \wedge v_{k,1} \wedge ... \wedge v_{k,d_k}|$ where $v_{i,1},...,v_{i,d_i}$ is an orthonormal basis in $V_i$; it is easily seen that the defined quantity is independent of the choices of orthonormal systems. 

The transversality and small support condition are quantified as follows.

\begin{con} \label{C} Assume that we are given submanifolds $S_i' \subset S_i, i =1,...,k$, of codimension $c_{i}$,
with the property that  there exists $d >0$ such that
\begin{equation} \label{normal6}
| N_{\zeta_1} S_1' \wedge  .. \wedge N_{\zeta_k} S_k' | \geq d
\end{equation}
for all choices $\zeta_i \in S_i'$. Given $g_i \in L^2(S_i), \forall i =1,...,k$, we assume that 
 $supp g_i \subset  B_{\mu_i} (S_i') \cap S_i$, where $0 < \mu_i \ll 1$.
 \end{con}
 
We make the observation that the total number of directions in which localization is provided cannot exceed $n-k$, that is $c_1+..+c_k \leq n-k$; this follows from \eqref{normal6}. 

From \cite{Be5} we have the following result.
\begin{theo} \label{MB}
Assume $S_i, i=1,...,k$ are smooth and  $g_1,...,g_k$ satisfy Condition \ref{C}.
Then for any $\epsilon > 0$, there is $C(\epsilon)$  such that 
the following holds true
\begin{equation} \label{Lf}
\| \prod_{i=1}^{k} \calE_i g_i \|_{L^\frac{2}{k-1}(B_R)} \leq C(\epsilon) (\prod_{j=1}^k \mu_j)^{\frac12} R^\epsilon \prod_{i=1}^{k} \| g_i \|_{L^2(S_i)}.
\end{equation}
\end{theo}
In this theorem, $B_R$ can be any ball of radius $R$ or any cube of side-length $R$. 

Assuming that, in addition to the conditions in the Theorem \ref{MB},  
the hypersurfaces $S_i, i=1,...,k$ satisfy \eqref{Scurv}, we let $C(\bm{\nu})=C(\nu_1,...,\nu_k)$ be the best constant such that 
\begin{equation} \label{LfT}
\| \prod_{i=1}^{k} \calE_i g_i \|_{\ell^{\frac{2}{k-1}}_{\mathcal{P}} L^{\frac{2}{k-1}}(Q_R)}\leq C(\epsilon) C(\bm{\nu}) (\prod_{j=1}^k \mu_j)^{\frac12} R^\epsilon \prod_{i=1}^{k} \| g_i \|_{L^2(S_i)}
\end{equation}
holds true for any $Q \in \mathcal{P}$ and $R \geq 1$, where $C(\epsilon)$ is the same constant used in Theorem \ref{MB}
and $\mathcal{P}$ is the family of parallelepipeds defined by \eqref{Pdef}; here $Q_R$
is the parallelepiped obtained by dilating $Q$ by a factor of $R$ from its center. 
By Theorem \ref{MB} we have the estimate
\begin{equation} \label{Cnu}
C(\bm{\nu}) \les \max(\nu_1^{-1},...,\nu_k^{-1})^\epsilon. 
\end{equation}
This estimate suffices for the purpose of obtaining the result claimed in this paper. However we conjecture that  $C(\bm{\nu}) \les 1$. In other words, the losses in $R$ should accumulate away from the base parallelepipeds $P \in \mathcal{P}$. It is easy to see that the estimate is lossless at the level of a single $P$ when $\mu_i \approx 1, i=1,..k$, that is one does not lose a factor of $\max(\nu_1^{-1},...,\nu_k^{-1})^\epsilon$ when estimating within a single $Q \in \mathcal{P}$. However generalizing for arbitrary $R$ and keeping track of the localization factors $\mu_i$ is not straightforward.

Second, we need an epsilon removal result for multilinear restriction estimates using the $\ell^p L^{\frac{2}{k-1}}$ quasi-norm, in the spirit of \cite{Tao-BR, Be5}. Unlike the result in Theorem \ref{Lf}, here we need curvature information.

\begin{lemma}\label{lem epsilon removal} Assume that the $S_i$ satisfy the geometric setup in Theorem \ref{mainTd} and let $p\geq\frac{2}{k}$. If for any $Q\in\mathcal{P}$ and $R>0$ we have:
\begin{align}
    \| \prod_{i=1}^k \calE_if_i \|_{\ell^p_{\mathcal{P}} L^{\frac{2}{k-1}}(Q_R)}\lesssim_\e R^\e C(\bm{\nu}) \prod_{i=1}^k \norm{f_i}_{L^2(d\sigma_i)}
\end{align}
then for all $q\geq p+(p+n+1)\frac{C}{\log(\frac{1}{\e})}$ (where $C=C(n)$) we have:
\begin{align}
    \| \prod_{i=1}^k \calE_if_i\|_{\ell^q_{\mathcal{P}} L^{\frac{2}{k-1}}(\R^n)}\lesssim C(\bm{\nu}) \prod_{i=1}^k \norm{f_i}_{L^2(d\sigma_i)}
\end{align}
\end{lemma}
Due to the usage of the quasi-norm and the degenerate curvature assumption \eqref{curvad}, we cannot directly use these aforementioned epsilon removal lemmas. Instead, in Section \ref{section 6}, we prove this lemma by generalizing the argument from \cite{Be5} to take into account the geometry of the parallelepipeds.

\subsection{Notation}

We use the standard notation $A \ls B$, meaning $A \leq C B$ for some universal $C$ which is independent of variables used in this paper. By $A \ls_N B$ we mean $A \leq C(N) B$ and indicate that $C$ depends on $N$. 

The Fourier transform of a Schwartz function $f:\R^n \rightarrow \C$ is defined by
\[
\mathcal{F}(f)(\xi)=\hat f(\xi)= \int e^{-i x \cdot \xi} f(x) dx. 
\]
The inverse Fourier transform is defined by 
\[
\mathcal{F}^{-1}(g)(x)=\check g(x)= \frac1{(2\pi)^n} \int e^{i x \cdot \xi} g(\xi) d\xi.
\]
These definitions are then extended to distributions, in particular to $L^p(\R^n)$ spaces, in the usual manner. 

If $v_1,...,v_m$ are vectors in $\R^n$, by $\operatorname{span}(v_1,...,v_m)$ we mean the standard subspace of $\R^n$ spanned by $v_1,...,v_m$.

We assume that the surfaces involved in the rest of this paper have small diameter in the sense that if $S=\Sigma(U)$ is a parametrization then $U \subset \R^{n-1}$ has small diameter in the classical sense. In particular, for every $i=1,...,k$ the following holds true:
\begin{equation} \label{cN}
|N_i(\zeta_1) - N_i(\zeta_2)| \leq c \nu_i, \quad \forall \zeta_1, \zeta_2 \in S_i,
\end{equation}
where $0 < c \ll 1$ is a constant small enough. This assumption does not affect our analysis as we can decompose each surface into finitely many pieces of small diameter, run the argument with the above setup, and then sum up the result on these pieces. 

For convenience, we recall the definition of the parallelepipeds from \eqref{Pdef}. 
\begin{align*}
    \mathcal{P}=\{ \prod_{i=1}^k [\nu_i^{-1}l_i, \nu_i^{-1}(l_i+1)] ; l=(l_1,...,l_n)\in \Z^n\}.
\end{align*}
Given a parallelepiped $P \in \mathcal{P}$ we define $c(P)=(\nu_1^{-1} (l_1 + \frac12),..., \nu_n^{-1}(l_n +\frac12))$ to be the center of $P$. For a given $R>0$ we also define $P_R$ to be the parallelepiped obtained by dilating $P$ from its center by a factor of $R$. We let $\mathcal{P}_R$ denote the set of such parallelepipeds at scale $R$, that is
\begin{align*}
    \mathcal{P}_R=\{ \prod_{i=1}^k [R \nu_i^{-1}l_i, R \nu_i^{-1}(l_i+1)] ; l=(l_1,...,l_n)\in \Z^n\}.
\end{align*}

\subsection{Acknowledgments}

We would like to thank Vitor Borges for helpful conversations, in particular in relation to our joint work in the companion paper \cite{BeBoCh}.

\subsection{AI declaration}

All the ideas and arguments in this paper were conceived and written by the authors. No AI was used in the writing or editing of this manuscript. The authors intend to use Gemini for additional spelling and grammar checks in an updated version of this paper.

\section{Tools from differential geometry} \label{tdg}

We first recall some basic facts about the shape operator that can be found in more detail in any classic differential geometry textbook, see for instance \cite{doCa}. 

Given a hypersurface  $S \subset \R^{n}=\{(\xi,\tau): \xi \in \R^{n-1}, \tau \in \R\}$ parametrized by $\tau=\varphi(\xi), \xi \in D \subset \R^{n-1}$, we define the Gauss map $g: S \rightarrow \S^{n-1} \subset \R^{n}$ ($\S^{n-1}$ is the unit sphere in $\R^{n}$) by 
$g(\zeta)=\frac{(-\nabla \varphi(\xi), 1)}{|(-\nabla \varphi(\xi), 1)|}$ where $\zeta=(\xi,\varphi(\xi)) \in S$. Since
$T_\zeta(S)$ and $T_{g(\zeta)} \S^{n-1}$ are parallel, we can identify them, and define $ d g_\zeta: T_\zeta S \rightarrow T_\zeta S$
by $d g_\zeta v= \frac{d}{dt} (N (\gamma(t)))|_{t=0}$ where $\gamma \subset S$ is a curve with $\gamma(0)=\zeta, \gamma'(0)=v$.  
The shape operator $S_{N(\zeta)}: T_\zeta S \rightarrow T_\zeta S$ is defined by 
\[
S_{N(\zeta)} = - d g_\zeta,
\]
 where we keep the subscript 
$N(\zeta)$ to indicate that the shape operator depends on the choice of the normal vector field at $S$. It is known that $S_{N(\zeta)}$ is symmetric, therefore there exists an orthonormal
basis of eigenvectors $\{e_i\}_{i=1,n-1}$ of $T_\zeta S$ with real eigenvalues $\{\lambda_i\}_{i=1,n-1}$. Locally $S$ is orientable and we assume
a consistency with the orientation in $\R^{n}$, that is $\{e_1,...,e_{n-1}\}$ is a basis in the orientation of $S$ and $\{e_1,...,e_{n-1},N(\zeta)\}$ 
is a basis in the orientation of $\R^{n}$. Then $e_i$ are the principal directions and $\lambda_i=k_i$ are the principal curvatures of $S$ (to be more precise they are the curvatures of the embedding $\on{id}: S \rightarrow \R^{n}$, where $\on{id}$ is the identity mapping). The Gaussian curvature is defined by $\det S_N=\prod_{i=1}^{n-1} \lambda_i$. 

As can be seen from the above notation, we use the letter $S$ to denote both hypersurfaces and the shape operator. 
However the shape operator notation will always occur together with the normal, that is $S_N$, and this clarifies any ambiguity in what $S$ stands for. 

Using the inclusion $i: S \rightarrow \R^n$, we obtain that $T_\zeta S \subset T_\zeta \R^n$ and by the standard identification $T_\zeta \R^n \cong \R^n$.  Using this identification allows us to write down various objects such as $N_1(\zeta_1) \wedge ... \wedge N_{k}(\zeta_k) \wedge S_{N(\zeta)} v$ that appears in \eqref{curva}. This identification will be tacitly used throughout the rest of the paper. 

Given any smooth hypersurface we note that for every $m \in \{1,...,n-1\}$, there exists $C_m$ such that the following holds true:
\begin{equation} \label{volc}
| S_{N(\zeta)} v_1 \wedge ... \wedge S_{N(\zeta)} v_m | \leq C_m | v_1 \wedge ... \wedge v_m |,
\end{equation}
for any $v_1,...,v_m \in T_\zeta S$. This follows from \eqref{smooth}; 
in fact $C_m$ depends only on the derivatives up to second order in \eqref{smooth}. 

If $S$ is a (more general) submanifold of $\R^n$, then we use $N_\zeta S$ to denote the normal plane at $\zeta \in S$. In the case when
$S$ is a hypersurface, $N(\zeta)$, the choice of unit normal to $S$ at $\zeta$ mentioned above, is the unit vector that spans $N_\zeta$. 
This hopefully clarifies any confusion that may occur due to the similarity in notation between the two entities.  

We recall the following supporting result from \cite{Be-6}. 

\begin{lemma}[{\cite[Lemma 2.1]{Be-6}}] \label{vols}
\quad
\begin{enumerate}[label=(\roman*)]
    \item Assume that $v_1,...,v_m \in \R^n$ are unit vectors such that $|v_1 \wedge .. \wedge v_m| \leq c^m$ for some $0 \leq c \leq 1$. Then there exists a unit vector $\alpha=(\alpha_1,...,\alpha_m) \in \R^m$ with the property that
    \begin{equation}
        |\sum_{i=1}^m \alpha_i v_i| \leq c.
    \end{equation}

    \item Assume that $A \in M_{n \times n}(\R)$ is a symmetric matrix with the property that $|\det A| \leq \epsilon$. Then there exists a unit vector $v \in \R^n$ with the property that $|Av| \leq \epsilon^\frac1n$. 
\end{enumerate}
\end{lemma}

In the introduction we have defined a special type of neighborhood for a submanifold;  given a submanifold $S' \subset \R^n$ of codimension $m$, 
we defined $B_\mu(S')$. The main result result of this section connects a localization property of the set of normals to a hypersurface to a localization of the hypersurface in a neighborhood of a submanifold.
The localization of the normals does not automatically imply localization of the hypersurface. Indeed, if the hypersurface is flat, then all normals are equal and no inference can be made. However if the localization comes in a ``direction of curvature", then we can deduce localization. This is quantified in the following result. 

\begin{lemma} \label{kL}
Let $S$ be a smooth hypersurface obeying \eqref{smooth} and of small diameter. We assume that there exists $\nu_0 >0$ and $c \approx 1$ such that
\begin{equation}
| S_{N(\zeta)}v | \leq \nu_0 |v|, \forall v \in T_\zeta S, \forall \zeta \in S.
\end{equation}
We assume that for every $\zeta \in S$ there exists a linear subspace $V_\zeta \subset T_\zeta S$  with the following property:
\begin{equation} \label{Sv}
|S_{N(\zeta)}v \wedge w| \geq c \nu_0 |v| \cdot |w|, \quad \forall v \in V^\perp_\zeta, w \in V_\zeta,  
\end{equation}
where $V^\perp_\zeta$ is the normal space to $V_\zeta$ inside $T_\zeta S$. 
Assume that $\calH \subset \R^n$ is a subspace of dimension $n-m$ with the property that
 \begin{equation} \label{ph}
| \pi_{V_\zeta^\perp} h - h| \ll  |h|, \quad \forall h \in \calH^\perp,
\end{equation}
where $\pi_{V_\zeta^\perp}$ is the orthogonal projection onto $V_\zeta^\perp$. Given $0 < \mu \ll 1$, 
 we define the following subset of $S$:
\begin{equation} \label{Nlocal}
\tilde S: = \{ \zeta \in S: | \pi_{\calH^\perp} N(\zeta)| \leq \nu_0 \mu  \}.
\end{equation}
Then there exists a submanifold  $\tilde S'$ of $S$ of codimension $m$ and $0 < \tilde \mu \les \mu$ with the following properties: 

\begin{enumerate}[label=(\roman*)]
    \item $\tilde S \subset B_{\tilde \mu}(\tilde S') \cap S$;

    \item the normal plane to $\tilde S'$ is transversal to $V_\zeta$ in the following sense: 
    \begin{equation} \label{NS}
        |N_\zeta \wedge V_\zeta| \ges 1, 
    \end{equation}
    for any $\zeta \in \tilde S'$.
\end{enumerate}

\end{lemma}

The subspaces $V_\zeta^\perp$ in the above statement should be thought as the natural replacement of $\calH^\perp$ in $T_\zeta S$; we need to work with such a substitute since the result would not be useful if we were to impose the condition $\calH^\perp \subset T_\zeta S, \forall \zeta \in S$. The quantification of the role of the parameter $0 < \nu_0 \les 1$ is important in this paper; the corresponding result in Lemma 2.2 in \cite{Be-6} contains the result when $\nu_0 \approx 1$.

\begin{proof} We start the argument with an important consequence of \eqref{Sv}. We fix $\zeta \in S$ and let $S_{N(\zeta)} V_\zeta^\perp \subset T_\zeta S$ be the image of of $V_\zeta^\perp$ under $S_{N(\zeta)}$; $S_{N(\zeta)} V_\zeta^\perp$ is itself a subspace. We will make use of two projectors: $\pi_{V_\zeta^\perp}: T_\zeta S \rightarrow V_\zeta^\perp$, the orthogonal projector onto $V_\zeta^\perp$, and  $\pi_\zeta: T_\zeta S \rightarrow S_{N(\zeta)} V_\zeta^\perp$, the orthogonal projection onto $S_{N(\zeta)} V_\zeta^\perp$.

We claim that there exists $c_1 > 0$ such that the following holds true:
\begin{equation} \label{wecon}
| \pi_{\zeta} e | \geq c_1 |e|
\end{equation} 
for any $e \in V_\zeta^\perp$. We argue by contradiction: assume \eqref{wecon} is false. This implies the existence of a unit vector $e \in V_\zeta^\perp$
such that $|e \cdot S_{N(\zeta)} v| < c_1 |S_{N(\zeta)} v| \leq c_1 \nu_0 |v|, \forall v \in V_\zeta^\perp$.  Using the symmetry of $S_{N(\zeta)}$ we obtain $| S_{N(\zeta)} e \cdot  v| <  c_1 \nu_0 |v|, \forall v \in V_\zeta^\perp$, which implies $|\pi_{V_\zeta^\perp} S_{N(\zeta)} e| < c_1 \nu_0$. Let $w = S_{N(\zeta)} e- \pi_{V_\zeta^\perp} S_{N(\zeta)} e$. This vector has the following properties:

\begin{itemize}
    \item $w \in V_\zeta$;
    
    \item Its size is bounded below by
    \[
    |w|^2= | S_{N(\zeta)} e|^2- |\pi_{V_\zeta^\perp} S_{N(\zeta)} e|^2 > \nu_0^2 - c_1^2 \nu_0^2,
    \]
    and above by
    \[
    |w|^2= | S_{N(\zeta)} e|^2- |\pi_{V_\zeta^\perp} S_{N(\zeta)} e|^2 \leq  | S_{N(\zeta)} e|^2 \leq \nu_0^2;
    \]

    \item It also satisfies
    \[
    | S_{N(\zeta)} e \wedge w|= |\pi_{V_\zeta^\perp} S_{N(\zeta)} e| \cdot |w| \leq c_1 \nu_0^2.  
    \]
\end{itemize}

Thus we obtain a contradiction with \eqref{Sv} provided that $c_1 \nu_0^2 <  \nu_0 ( \nu_0^2 - c_1^2 \nu_0^2)^\frac12$; obviously this is achievable if $c_1 >0$ is chosen small enough in \eqref{wecon}. 

We now return to the proof of the main statements in our theorem. We make two simplifying assumptions: there exists $\zeta_0 \in S$ such that $N(\zeta_0) \in \calH$ and that $\calH^\perp \subset V_{\zeta_0}^\perp$. If the first assumption fails, we modify $\mathcal{H}$ by $O(\nu_0 \mu)$ to achieve it (unless the set $\tilde S$ is void in which case there is nothing to prove). If the second assumption fails we work with 
$\pi_{V_{\zeta_0}^\perp} e_1,..., \pi_{V_{\zeta_0}^\perp} e_m$ in the argument below instead of $ e_1,..., e_m$ . 

Let $e_1,...,e_m$ be an orthonormal base in $\calH^\perp$. Since $N_0=N(\zeta_0) \in \calH$ and $e_1,...,e_m$ are all perpendicular to $\calH$, they are perpendicular to $N(\zeta_0)$ and therefore $e_1,...,e_m \in T_{\zeta_0} S$. 

Our first claim is that $(\calN)^{-1}(\calH):= \{ \zeta \in S: N(\zeta) \in \calH \}$ is a submanifold of $S$ of codimension $m$, thus effectively identifying $\tilde S'$.  To do so it is convenient to work in a system of coordinates. 

We let $\pi_0=\pi_{\zeta_0}: T_{\zeta_0} S \rightarrow S_{N(\zeta_0)} V_{\zeta_0}^\perp$ be the orthogonal projection defined above. We first note the following
\begin{equation} \label{pi0e}
|\pi_0 e_i \wedge \calH | \ges 1, \quad \forall i=1,...,m.
\end{equation}
If this were not the case, then, since $e_i \in \calH^\perp$, $|\pi_0 e_i \cdot e_i| \ll 1$ and this would be in contradiction with \eqref{wecon}. 

Now, for every $i \in \{1,...,m \}$ we have that $\pi_0 e_i \in  S_{N(\zeta_0)} V_{\zeta_0}^\perp$, therefore there exists $\tilde e_i \in V_{\zeta_0}^\perp$ with the property
$S_{N(\zeta_0)} \tilde e_i = \nu_0 \pi_0 e_i$; we add the factor $\nu_0$ so that $|\tilde e_i| \approx 1$, which we prove below. At first $\tilde e_i$ is uniquely determined modulo $\ker S_{N(\zeta_0)}$; however 
\eqref{Sv} implies $\ker S_{N(\zeta_0)} \cap V_{\zeta_0}^\perp=\emptyset$, and this makes the choice of $\tilde e_i$ unique within $V_{\zeta_0}^\perp$.

Next we establish that the system $\{\tilde e_i\}_{i=1,...,m}$ has good transversality properties. From \eqref{Sv} we have $c \nu_0 |\tilde e_i| \leq |S_{N_{\zeta_0}} \tilde e_i|= \nu_0 |\pi_0 e_i| \leq \nu_0 |e_i|= \nu_0$, therefore
\begin{equation} \label{elen}
|\tilde e_i| \leq c^{-1}, \quad \forall i=1,...,m.
\end{equation}
The set of vectors $\{ \tilde e_i\}_{i=1,...,m}$ is transversal in the following sense: there exists $\tilde c > 0$ (universal, independent of $\nu_0$) such that 
\begin{equation} \label{etrans}
|\tilde e_1 \wedge .. \wedge \tilde e_m| \geq \tilde c. 
\end{equation}
 We assume the above is false, that is $|\tilde e_1 \wedge .. \wedge \tilde e_m| < \tilde c$ and we establish that if $\tilde c$ is chose small enough then we run into a contradiction. We claim that 
 \[
 |S_{N(\zeta_0)}\tilde e_1 \wedge .. \wedge S_{N(\zeta_0)} \tilde e_m| \les  \nu_0^m |\tilde e_1 \wedge .. \wedge \tilde e_m|. 
 \]
 This is a consequence of the fact that $\nu_0^{-1} S_{N(\zeta_0)}: V_{\zeta_0}^\perp \rightarrow S_{N(\zeta_0)} V_{\zeta_0}^\perp$ is a linear operator whose matrix has determinant $\approx 1$; indeed we have $|\nu_0^{-1} S_{N(\zeta_0)} v| \approx |v|, \forall v \in V_{\zeta_0}^\perp$. 
 
  Therefore $|S_{N(\zeta_0)}\tilde e_1 \wedge .. \wedge S_{N(\zeta_0)} \tilde e_m| \les \tilde c \nu_0^m$, and 
 $|\pi_0 e_1 \wedge .. \wedge \pi_0 e_m| \les \tilde c$. Using \eqref{wecon} we obtain
 $|\frac{\pi_0 e_1}{|\pi_0 e_1|} \wedge .. \wedge \frac{\pi_0 e_m}{|\pi_0 e_m|}|  \les \tilde c$. Using Lemma \ref{volc} it follows that there is a unit vector $\alpha \in \R^m$ such that the size of
 $e=\alpha_1 \frac{e_1}{|\pi_0 e_1|} +..+\alpha_k \frac{e_m}{|\pi_0 e_m|}$ satisfies $|\pi_0 e| \les \tilde c^\frac1m$.
  On the other hand, 
 \[
 |e|^2= \sum_{i=1}^m \frac{\alpha_i^2}{|\pi_0 e_i|^2} \geq  1,
 \]
 while from \eqref{wecon} it follows that $|\pi_0 e| \ges 1$. Thus if $\tilde c$ is small enough we obtain a contradiction. This establishes \eqref{etrans}. 
  
 \eqref{elen} and \eqref{etrans} imply that  we have produced a system of vectors with $|\tilde e_i| \approx 1, i=1,...,m$ and $|\tilde e_1 \wedge .. \wedge \tilde e_m| \ges 1$. We can complete
this with a basis $\{\tilde e_i\}_{i=1,...,n-1}$ in $T_{\zeta_0} S$ with similar properties; indeed take $\{e_{m+1},...,e_{n-1}\}$ to be an orthornormal basis
to the $(span(\tilde e_1,...,\tilde e_m))^\perp$, the normal subspace to the subspace spanned by $\tilde e_1,..., \tilde e_m$. 

Since $N(\zeta_0) \perp T_{\zeta_0} S$, we also obtain a basis of vectors in $\R^n$ by adding $N_0=N(\zeta_0)$ to the system $\{\tilde e_i\}_{i=1,...,n-1}$. 
We let $\pi_{N_0}: \R^{n} \rightarrow \R^{n-1}$ be the projector along $N_0$ onto $\R^{n-1}$ which is identified with $T_{\zeta_0} S$ and where
$\{\tilde e_i\}_{i=1,...,n-1}$ forms a basis. We define the parametrization of $S$ by simply projecting $S$ onto $\R^{n-1}$ along $N_0$;  this is a good parametrization in that it obeys \eqref{smooth}. We let $U$ be the image of $S$ under this parametrization, that is $U=\pi_{N_0}(S)$. In $\R^{n-1}$ we use coordinates $(\xi_1,...,\xi_{n-1})$ with the convention that $\xi_{1}, .., \xi_{n-1}$ are the coordinates in the directions of $\tilde e_1,..., \tilde e_{n-1}$ respectively; a apparent drawback of this system of coordinates is that is not Cartesian since $\tilde e_1,..., \tilde e_m$ is not an orthonormal system, but this plays no role in our analysis.  We define 
$\xi_0=\pi_{N_0}(\zeta_0)$. The parametrization of $S$ is given by $\Sigma: U \rightarrow \R^n$ where $(\pi_{N_0} \circ \Sigma) (\xi)=\xi, \forall \xi \in U$. 
$(\calN)^{-1}(\calH)$ is given by the set of equations
\[
N(\Sigma(\xi))\cdot e_i=0, \quad \forall i=1,...,m. 
\]
We define $F: U \subset \R^{n-1} \rightarrow \R^m, F=(F_1,...,F_m)$ by $F_i(\xi): = N(\Sigma(\xi)) \cdot e_i, \forall i=1,...,m$ 
and aim to solve the equation $F(\xi)=0$. This is done using the implicit function theorem and for this purpose we compute
\[
\frac{D(F_1,...,F_m)}{D(\xi_1,...,\xi_m)} (\xi_0)= \det [S_{N(\Sigma(\xi_0))} \tilde e_i \cdot e_j]_{i,j=1,...,m} = \det [\pi_0 e_i \cdot e_j]_{i,j=1,...,m}.
\]
The matrix $[\pi_0 e_i \cdot e_j]_{i,j=1,...,m}$ is symmetric because $\pi_0$ is a symmetric operator. Our goal is to prove that 
\[
|\det [\pi_0 e_i \cdot e_j]_{i,j=1,...,m}| \ges 1. 
\]
Assume by contradiction that $|\det [\pi_0 e_i \cdot e_j]_{i,j=1,...,m}| \ll 1$. 
By Lemma \ref{vols} part ii), there exists $\alpha=(\alpha_1,...,\alpha_m) \in \R^m$ a unit vector such that 
\[
|(\sum_{i=1}^m \alpha_i \pi_0 e_i \cdot e_j)_{j=1,...,m}| \ll 1.
\]
We let $e=\sum_{i=1}^m \alpha_i  e_i$ be the unit vector in $(\calH)^\perp$ and conclude that $|(\pi_0 e \cdot e_j)_{j=1,...,m}| < \tilde \nu_3^\frac1m$.
From this we further estimate
\[
|\pi_0 e \cdot e| \leq \sum_{i=1}^m |\alpha_i| |\pi_0 e \cdot e_i| \leq |\alpha| |(\pi_0 e \cdot e_j)_{j=1,...,m}| \ll 1,
\]
and conclude with $|\pi_0 e|^2 = |\pi_0 e \cdot \pi_0 e|= |\pi_0^2 e \cdot  e|= |\pi_0 e \cdot  e | \ll 1$. On the other hand,
from \eqref{wecon} we have $|\pi_0 e|^2 \ges  |e|^2=1$, thus we obtain a contradiction.  

We can invoke the implicit function theorem to conclude that there exists an open neighborhood $U' \subset U$ of $\xi_0$ and 
$\Sigma': U' \cap \R^{n-m-1} \rightarrow \R^m$ such that
\[
\Sigma^{-1} (\calN)^{-1}(\calH) \cap U'=\{ (\Sigma'_1(\xi_{m+1},...,\xi_{n-1}), .., \Sigma'_m(\xi_{m+1},...,\xi_{n-1}), \xi_{m+1},...,\xi_{n-1}\}.
\]
This gives us the correct structure for $S'=\Sigma^{-1} (\calN)^{-1}(\calH) \cap U'$. 
The size of the neighborhood $U'$ of $\xi_0$ is uniform with respect to the choice of $\xi_0$ since it depends on parameters for which we have uniform bounds across $S$, see \eqref{smooth}, \eqref{Sv} and \eqref{wecon}; to be more precise, the result gives $U'=B_r(\xi_0) \cap U$ and the size of $r$ is uniform with respect to the choice of $\xi_0$. If we assume that the diameter of $S$ is small enough, then $U=\pi_{N_0}(S) \subset B_r(\xi_0)$
and the local solution becomes the global one.

From the above analysis it also follows that $\tilde S'=\Sigma(S')$ is a submanifold of $S$ of codimension $m$.

Next we prove that $\tilde S$ belongs to the correct neighborhood of $\tilde S'$. To do so, it suffices to prove that $\pi_{N_0} \tilde S \subset B_{\tilde \mu} S'$,
for some $\tilde \mu \les \mu$. We rely on the fact that the system $\{ \tilde e_i \}_{i=1,...,m}$ is transversal to $S'$.  We fix $\xi \in S'$, fix $i \in \{1,...,m\}$, and let $\gamma(t)=\xi+t \tilde e_i \in U, |t| \ll 1$; the range of $t$ is constrained by the size of $U$
since we impose $\gamma(t) \in U$. We have
\[
\begin{split}
N(\Sigma(\gamma(t)))-N(\Sigma(\gamma(0))) & = \int_{0}^{t} S_{N(\gamma(s))} d\Sigma(\gamma(s)) \gamma'(s) ds  \\
& = t S_{N(\gamma(0))} d\Sigma(\gamma(0)) \gamma'(0) + \nu_0 o(t) \\
&= t S_{N(\zeta_0)} d\Sigma(\xi_0) \tilde e_i + \nu_0 o(t)  \\
&=  t \nu_0 \pi_0 e_i + \nu_0 o(t). 
\end{split}
\]
In the above we used the following straightforward facts:  $d\Sigma(\xi_0)=Id$, $\|S_{N(\gamma(0))}-S_{N(\zeta_0)} \| \ll \nu_0$ (a consequence of \eqref{sm23}) and 
$\|d\Sigma(\gamma(0))-d\Sigma(\xi_0)\| \ll 1$; the last two inequalities involve matrix norms and use the small diameter property of $S$. 

From \eqref{wecon} and \eqref{pi0e}, we have that $|\pi_0 e_i| \ges 1$ and $| \pi_0 e_i \wedge \calH| \ges 1$; since $N(\Sigma(\gamma(0))) \in \calH$,
it follows that a necessary condition of $\Sigma(\gamma(t)) \in \tilde S$ is that $|t| \les \mu $. 
Since the system $\{ \tilde e_i \}_{i=1,...,m}$ is transversal to $S'$, whose co-dimension is $m$, it follows that $\pi_{N_0} \tilde S \subset B_{\tilde \mu} S'$,
for some $\tilde \mu \les \mu$.

As for ii), we notice that any unit vector $n \in N_\zeta$ has the property that its projection on the $span\{ \tilde e_1,...,\tilde e_m \} \subset V_{\zeta_0}^\perp$ 
has length bounded from below by a universal constant; this simply follows from the structure of $\tilde S'$ and the smallness of the diameter of all submanifolds involved. As a consequence $|n \wedge V_{\zeta_0}^\perp| \ges 1$, thus $|n \wedge V_{\zeta}^\perp| \ges 1$. This implies that $|N_\zeta \wedge V_\zeta| \ges 1$. 

\end{proof} 

We end this section with the proof of Corollary \ref{cor qnls}. We do so here because we prepared all the necessary geometrical background in this section.

\begin{proof}[Proof of Corollary \ref{cor qnls}]
The estimate follows immediately from Theorem \ref{mainTd}, once we check the curvature conditions \eqref{Scurv} and \eqref{curvad} with the correct numerology.

It is convenient to use spherical coordinates $(r,\theta,\phi)$ in this argument, due to the radial symmetry of the paraboloid. In particular, the paraboloid is parametrized as
\begin{align*}
    \varphi(r,\theta,\phi)=(r\cos(\theta)\sin(\phi), r\sin(\theta)\sin(\phi), r\cos(\phi), r^2)\in\R^4.
\end{align*}
The tangent plane at any point $\zeta=(r,\theta,\phi,r^2)$ has unit basis vectors
\begin{align*}
    v_r &= (\cos(\theta)\sin(\phi), \sin(\theta)\sin(\phi), \cos(\phi), 2r)\cdot\frac{1}{(4r^2+1)^{\frac{1}{2}}},\\
    v_\theta &= (-\sin(\theta), \cos(\theta), 0, 0),\\
    v_\phi &= (\cos(\theta)\cos(\phi), \sin(\theta)\cos(\phi), -\sin(\phi), 0).
\end{align*}
The shape operator in this basis is then given by the diagonal matrix
\begin{align*}
    S_{N(\zeta)}=\begin{pmatrix}
        -\frac{2}{(4r^2+1)^{\frac{3}{2}}} & 0 & 0 \\
        0 & -\frac{2}{(4r^2+1)^{\frac{1}{2}}} & 0 \\
        0 & 0 & -\frac{2}{(4r^2+1)^{\frac{1}{2}}}
    \end{pmatrix}.
\end{align*}

For fixed $i \in \{1,2,3\}$ checking the condition \eqref{curvad} amounts to establishing 
\begin{equation} \label{curvda3}
|\la S_{N(\zeta_i)} v, v \ra| \ges \nu_i |v|^2, 
\end{equation}
$\forall v \in T_{\zeta_i} S_i$ with the property $v \perp N_j, j \in \{1,2,3\} \setminus \{ i \}$, for some choice of normals $N_j$ at the surfaces $S_j$. 

If $i=1,2$ then we see that 
\[
|\la S_{N(\zeta_i)} v, v \ra| \approx |v|^2,
\]
because in the support of $S_1,S_2$ we have $|r| \les 1$. Thus we obtain \eqref{curvda3} with $\nu_1 \approx \nu_2 \approx 1$. 

Consider now the case $i=3$. If $M \approx 1$, then the same argument above applies. If $M \gg 1$ then we argue as follows. We pick $v \in T_{\zeta_3} S_3$ as detailed in \eqref{curvda3} and make the observation that since $v \perp N_j, j=1,2$ for some choices of normals $N_j$ to $S_j$, it follows that if $v=c_r v_r + c_\theta v_\theta + c_\phi v_\phi$, then we have $|c_\theta| + |c_\phi| \ges |c_r|$. Indeed, if this were not the case, that is $|c_\theta| + |c_\phi| \ll |c_r|$, then using the fact that $r \gg 1$ and $N_1=\frac{(-2 \xi,1)}{\sqrt{1+4\xi^2}}$, it follows that
$|v_r \cdot N_1| \ges 1$ and from this we obtain
\[
|v \cdot N_1| \ges |c_r| |v_r \cdot N_1| \ges 1,
\]
thus $v$ cannot be orthogonal to any normal $N_1$ at $S_1$. Using that $|c_\theta| + |c_\phi| \ges |c_r|$ we then conclude
\[
|\la S_{N(\zeta_3)} v, v \ra| \approx r^{-3} |c_r|^2 + r^{-1}(|c_\theta|^2 + |c_\phi|^2) \ges r^{-1} |v|^2.
\]
Thus we obtain \eqref{curvda3} with $\nu_3 \approx 2^{-k}$.
\end{proof}

\section{Wave packets} \label{swp}

In this section we introduce the wave packet theory for operators of type $\calE$. At its core, we use the setup from previous works of the first author, see \cite{Be2} for more details, which was inspired by the work of Tao in \cite{Tao-BW}. We make specific choices for the parameters used so as to (almost) match the construction that Guth uses in \cite{Gu-II}. There is one distinct feature in this work related to the overall degenerate curvature assumption in \eqref{Scurv}.  

Consider the extension operator $\calE f$ associated to the hypersurface $S$ that is parametrized by $\Sigma: U \subset \R^{n-1} \rightarrow \R^n$, where $\Sigma$ is smooth in the sense of \eqref{smooth}. Without restricting the generality of the argument, we assume that the parametrization is given by $\Sigma(\xi)= (\xi,\varphi(\xi))$. We also assume that we have the global bound
\begin{equation} \label{Scurv2}
| S_{N(\zeta)} v | \leq \nu |v|, \forall v \in T_{\zeta} \Sigma,
\end{equation}
for some $0 < \nu \les 1$. The only noticeable difference between the wave packet theory below and the standard one (corresponding to $\nu \approx 1$) is the effect of the parameter $\nu$ in the geometry of the tubes where the wave packets are essentially supported. 

We choose a parameter $0 < \delta \ll 1$ to be specified later. Given $R>0$, the standard wave packet decomposition 
is adapted to a ball of radius $R$, or to a strip of length $R$ in the $x_n$ direction, this being the direction of wave propagation. In our setup the wave packet decomposition will be adapted to strips of length $\nu^{-1} R$ in the $x_n$ direction; the initial construction below is in a strip $|x_n| \leq \nu^{-1} R$.  

We begin with  the two lattices: $\mathcal{L}=R^{-\frac12} \Z^{n-1} \cap U$ and $L= R^{\frac{1+\delta}2} \Z^{n-1}$. Both of these lattices are built along the directions of the standard orthonormal system of vectors $e_1,...,e_{n-1}$ in $\R^{n-1}$. 

We decompose $U$ as follows: 
\[
U= \bigcup_{\xi_0 \in \mathcal{L}} A_{\xi_0}
\]
where $A_{\xi_0}$ consists of the points in $U$ that are closer to $\xi_0$ than any other elements of $\mathcal{L}$. 
Therefore $A_{\xi_0}$ belongs to the $O(R^{-\frac12})$ neighborhood of $\xi_0$.  We have the following decomposition of $f$:
\[
f=\sum_{\xi_0 \in \mathcal{L}} \chi_{A_{\xi_0}} f. 
\]
The terms above have good frequency support and next we proceed with the spatial localization. Let $\eta_1:\R \rightarrow [0,+\infty)$ be a Schwartz function, normalized in $L^1$, that is $\| \eta_1 \|_{L^1}=1$, and with Fourier transform supported in $(-1,1)$. Given some $r>0$ we denote by $\eta_r(x)=r^{-1} \eta_1(r^{-1} x)$ and note that $\hat \eta_r$ is supported in $(-r^{-1},r^{-1})$. Based on this function, we can construct similar ones acting on any $\R^{n}$ space and with similar properties; we will abuse notation and use the same $\eta_1$ notation. Thus we define $\eta_1(x_1,...,x_n)=\prod_{i=1}^n \eta_1(x_i)$. As above we define $\eta_r(x)=r^{-n} \eta_1(r^{-1} x)$ and note that $\hat \eta_r$ is supported in the $(-r^{-1},r^{-1})^{n}$. 

For each $x_0 \in L$, 
\[
\eta^{x_0}_{R^\frac{1+\delta}2}(x) = \eta_{R^\frac{1+\delta}2} (x-x_0)
\]
and notice that, by the Poisson summation formula and properties of $\eta_0$, 
\begin{equation} \label{pois}
\sum_{x_0 \in L} \eta^{x_0}_{R^\frac{1+\delta}2} =1.
\end{equation}
We define $f_{x_0,\xi_0}= \hat \eta^{x_0}_{R^\frac{1+\delta}2}  \ast \chi_{A_{\xi_0}} f$; from the above considerations we obtain the following decomposition of $f$:
\[
f=\sum_{x_0 \in L, \xi_0 \in \mathcal{L}} f_{x_0,\xi_0}. 
\]
With $x_T \in L, \xi_T \in \mathcal{L}$ we define the tube $T:=\{ x=(x',x_n) \in \R^n: |x'-x_T + \nabla \varphi(\xi_T) x_n| \leq  R^{\frac12+\delta}, |x_n| \leq  \nu^{-1} R \}$ and denote by $\mathcal{T}$ the set of such tubes. We remark that the only difference between the standard wave packet theory (say for the Schr\"odinger equation) and what we use here is the factor of $\nu^{-1}$
appearing in the definition of tubes. If $\nu \ll 1$, then the wave packets are ``longer" which is due to the slower dispersion as a consequence of \eqref{Scurv2}.

For each $T$ as above, we define $v(T)=\frac1{\sqrt{1+|\nabla \varphi(\xi_T)|^2}}(-\nabla \varphi(\xi_T),1)$, the unit vector describing the direction of the center line of $T$. We let $f_T=f_{x_T,\xi_T}$ and recall that we have the decomposition
\begin{equation} \label{lind}
f= \sum_{T \in \mathcal{T}} f_{T}.
\end{equation}
The relevance of the tubes is that $\calE f_T$ is essentially localized within the tube $T$ (in fact in a thinner tube $T':=\{ x=(x',x_n) \in \R^n : |x'-x_T + \nabla \varphi(\xi_T) x_n| \leq  R^{\frac{1+\delta}2}, |x_n| \leq \nu^{-1} R \}$ of width $R^{\frac{1+\delta}2}$) and decays very fast outside $T$, that is
\[
\| \calE f_T\|_{L^\infty(T^c)} = O(R^{-N}). 
\]
This is a consequence of the more general inequality:
\begin{equation} \label{decay}
| \calE f_T(x)|  \ls R^{-\frac12} (1+\frac{d(x,T')}{R^{\frac{1+\delta}2}})^{-N} \|f_T\|_{L^2},
\end{equation}  
whose proof can be found in \cite{Be2}. We also recall the standard orthogonality property that this decomposition has: 
if $\calT' \subseteq \calT$ the following holds true:
\begin{equation} \label{sel}
\|  \sum_{T \in \calT'} f_T  \|^2_{L^2} \les   \sum_{T \in \calT'} \| f_T  \|^2_{L^2}. 
\end{equation}
We recall that all the above properties are meant to be read in the spatial regime $|x_n| \leq \nu^{-1} R$. By re-centering the construction at other points $p \neq 0$, the wave packet decomposition can be adapted to other strips $|x_{n}-p| \leq \nu^{-1} R$. 

In practice we have our extension operators $\calE_i f_i, i=1,..., n-1$ and each of them has a corresponding wave packet decomposition on the strip $|x_i| \leq \nu_i^{-1} R$. This creates a universal wave packet decomposition adapted to the parallelepiped $P_{2R}(0)=\{x \in \R^n: |x_i| \leq \nu^{-1}_i R, i=1,...,n\}$ centered at the origin. The above construction can be adapted to any translation of this parallelepiped, that is to any other $P_{2R}$ whose center is different than $0$. The precise use of the wave packet family will be clear from the context. For convenience, we will restrict this construction to the strips $0 \leq x_i \leq \nu_i^{-1} R$ instead, which just avoids the (harmless) extra factor of $2$ in the scale of the parallelepipeds.

\section{Tools and results from algebraic geometry} \label{TAG}

In this section we collect all the tools from algebraic geometry that we employ in this paper. The algebraic geometry setup here follows closely the one developed by Guth in \cite{Gu-I,Gu-II} - for more details on our exposition we refer the reader to \cite{Gu-II}.  

The first concept that we introduce is that of a real algebraic variety: this is the locus of common zeros of a collection of polynomials. 
More precisely, an algebraic variety
takes the form $Z(P_1,...,P_d)=\{ x \in \R^n: P_i(x)=0, \forall i=1,...,d\}$; here $P_i \in \R[X_1,...,X_n]$ are real valued polynomials. The algebraic varities do not necessarily have good geometric properties, in particular they may lack a good manifold structure. To avoid this problem, a more restrictive class is introduced. The variety $Z(P_1,...,P_{m})$ is said to be a transverse complete intersection if
\begin{equation} \label{ci}
\nabla P_1(x) \wedge ... \wedge \nabla P_{m}(x) \neq 0, \qquad \forall x \in Z(P_1,...,P_{m})
\end{equation}
A transverse complete intersection $Z(P_1,...,P_{m})$ is a smooth $n-m$-dimensional manifold. In particular, if $n=m$, we have the following result. 
\begin{theo} [{\cite[Theorem 5.8]{Gu-II}/\cite[Theorem 5.2]{CKW}}] \label{ftci}
A transverse complete intersection $Z(P_1,...,P_{n})$ is a finite set of cardinality at most  $\prod_{i=1}^n deg P_i$. 
\end{theo}

From Sard's theorem it follows that
 generically the set of common zeroes of a family of polynomials is a transverse complete intersection; the precise statement is made below. 

\begin{lemma}[{\cite[Lemma 5.1]{Gu-II}}] \label{lci}
If P is a polynomial on $\R^n$, then for almost every $c_0 \in \R$, $Z(P+c_0)$ is a transverse complete intersection.
More generally, suppose that $Z(P_1,...,P_{m})$ is a transverse complete intersection and that $P$ is another polynomial. Then for almost every $c_0 \in \R$, 
$Z(P_1,...,P_{m},P+c_0)$ is a complete transverse intersection. 
 \end{lemma}

The standard polynomial partitioning result used in restriction theory is the following.
\begin{theo}[{\cite[Theorem 5.5]{Gu-II}}] \label{PP}
Suppose that $W \geq 0$ is a non-zero $L^1$ function on $\R^n$. Then, for any degree $D $ the following holds:

There is a sequence of polynomials $Q_1,...,Q_S$, with $\sum_{s=1}^S \on{deg} Q_s \les D$ and $2^S \approx D^n$, satisfying the following properties.
Let $P=\prod_{s=1}^S \tilde Q_s= \prod_{s=1}^S (Q_s+c_s)$ where $c_s \in \R$. Let $O_i$ be the open sets given by the sign conditions of $\tilde Q_s$, that is
\[
O_i=\{ x \in \R^n| \operatorname{sign}(\tilde Q_s(x))=\sigma_s \}
\]
where $\sigma_s \in \{-1,+1\}$. There are $2^S \approx D^n$ cells $O_i$ and $\R^n \setminus Z(P)=\bigcup_i O_i$. 

If the constants $c_s$ are sufficiently small, then for every $O_i$ 
\[
\int_{O_i} W \approx D^{-n} \int_{\R^n} W.
\]
\end{theo}
From Lemma \ref{lci} it follows that for almost every $c_s$, $Z(\tilde Q_s)$ is a complete intersection for each $s$, which in turn implies that $Z(P)$ is generically a finite union of transverse complete intersections. More generally, if $Z(P_1,...,P_{m})$ is a transverse complete intersection and $m \leq n-1$,
then for a generic choice of constants $c_s$, $Z(P_1,...,P_{m}, \tilde Q_s)$ is a transverse complete intersection for every $s$. 

One can rephrase the result above in terms of general measures in which case the final partitioning property is a consequence of the additivity of the measure. In particular, if we let $\mu(A)=\int_A W$ for every Lebesgue measurable set $A$, then $\mu$ is additive in the sense that $\mu(A \cup B)= \mu(A) + \mu(B)$ for any measurable sets $A,B$ with $A \cap B=\emptyset$ and, in that case, the result can be stated as
\[
\mu(O_i) \approx D^{-n} \mu(\R^n). 
\]

In this paper we need an analogue of the above theorem to be applied to the $\ell^p_{\mathcal{P}}L^{\frac{2}{k-1}}(P)$ quasi-norm. To do this, we first note the generalized ham sandwich theorem for submeasures proved by Stone and Tukey \cite{StTu} (written in the form used by Guth in \cite{Gu-I}):

\begin{lemma}
Let $\mu_j$ be a submeasure on $\R^n$ for each $1\leq j\leq N$. Let $V$ be a vector space of continuous functions on $\R^n$ such that for each nonzero $f\in V$ and $1\leq j\leq N$ we have $\mu_j(Z(f))=0$. If $N<\dim(V)$ then there exists a nonzero $v\in V$ such that for each $j$ we have:
\begin{align*}
    \mu_j(\{v>0\})=\mu_j(\{v<0\})
\end{align*}
\end{lemma}

Using this, we then have the following partitioning lemma. We do not claim that this result is new but we were not able to find the exact statement we needed in the literature.

\begin{theo} \label{PP-om}
Let $\mu$ be a submeasure. Then the conclusion of Theorem \ref{PP} holds true with the adjustment
\[
\mu(O_i) \ges D^{-n} \mu(\R^n). 
\]
\end{theo}

\begin{proof}
Essentially the same iteration argument works as in the additive case using the polynomial ham sandwich theorem. The only difference is that at each step in the iteration we use the subadditivity $\mu(\R^n) \leq \mu(O)+\mu(O^c)$ and the bisection property $\mu(O)=\mu(O^c)$ to deduce that $\mu(O) \geq \frac12 \mu(\R^n)$ and $\mu(O^c) \geq \frac12 \mu(\R^n)$. This is then carried through  every step in the iteration and leads to the desired conclusion.
\end{proof}

The next result is helpful in controlling the tangent plane of a variety. Here we assume that the $n-m$-dimensional variety $Z=Z(P_1,...,P_{m})$ is a transverse complete intersection and that $m \leq n-1$. Given $w \in \Lambda^{n-m} \R^n$, we define $Z_w$ by
\begin{equation} \label{Zw}
Z_w: =\{ x \in Z | \nabla P_1(x) \wedge ... \wedge \nabla P_{m}(x) \wedge w = 0\}. 
\end{equation}
The expression $g_w(x):= \nabla P_1(x) \wedge ... \wedge \nabla P_{m}(x) \wedge w$ is a polynomial with degree at most $\on{deg} P_1 + ...+\on{deg} P_{m}$, therefore $Z_w$ is an algebraic variety. The following result states that, generically, $Z_w$  is a a smooth transverse complete intersection.

\begin{lemma}[Lemma 5.6, \cite{Gu-II}] \label{Lw}
For almost every $w \in \Lambda^{n-m} \R^n$, $Z_w= Z(P_1,...,P_{m},g_w)$ is a smooth transverse complete intersection. 
\end{lemma}

The last result we need from \cite{Gu-II} provides a useful tool for controlling the transverse intersections between a tube and a variety.
Given a direction $v_0 \in \R^n$ with $|v_0|=1$, a radius $r$ and a point $x_0 \in \R^n$, we let $T(x_0,v_0,r)=\{x_0+tv_0+e: t \in \R, e \in \R^n, |e| \leq r \}$
be the tube of radius $r$ that passes through $x_0$ and has direction $v_0$.

Suppose $Z \subset \R^n$ is a transverse complete intersection; we define the following two sets:
\[
Z_{v_0, \leq \alpha}= \{z \in Z: \angle(v_0, T_z Z) \leq \alpha\}, \qquad  Z_{v_0, > \alpha} = Z \setminus Z_{v_0, \leq \alpha}.
\]
We have the following result. 
\begin{lemma}[Lemma 5.7, \cite{Gu-II}] \label{Zan}
Suppose that $Z=Z(P_1,...,P_{m})$ is a transverse complete intersection and that the polynomials $P_i, i=1,...,m$ have degrees at most $D$. Let 
$T=T(x_0,v_0,r)$ be a tube as above. Then for any $\alpha >0$, $Z_{v_0,>\alpha} \cap T$ is contained in a union of $\les D^n$ balls of radius $\les r\alpha^{-1}$.
\end{lemma}

\section{Main Argument} \label{sma}

Given $Q \in \mathcal{P}$ and $R > 0$, we recall that by $Q_R$ the dilation of $Q$ by a factor of $R$ from its center. Our goal is to prove the following local estimate.
\begin{prop} \label{indp}
Assume the setup in Theorem \ref{mainTd} with $p=\frac{2(n+k)}{k(n+k-2)}$.
Then for any $\epsilon>0$ the following holds true: for any $R >0$ and $Q \in \mathcal{P}$ there exists $C(\e)$ such that
\begin{equation}  \label{mrei-ind}
\left( \sum_{P \in \mathcal{P}} \| \prod_{i=1}^k \calE_i f_i \|^p_{L^\frac2{k-1}(P \cap Q_R)} \right)^\frac1p \lesssim C(\epsilon) C(\bm{\nu}) R^{ \epsilon} 
 \prod_{i=1}^k \| f_{i} \|_{L^2(U_i)}.  
\end{equation}
\end{prop}
Indeed, by invoking Lemma \ref{lem epsilon removal} we can remove the $\epsilon$-loss above in the regime $p>\frac{2(n+k)}{k(n+k-2)}$ and obtain Theorem \ref{mainTd}. The proof closely follows the one in \cite{Be-6}, which in turn uses an induction on scales type argument that uses polynomial partitioning along the lines developed by Guth in \cite{Gu-II}.

We first fix some notation. Let $A(R)$ be the best constant for which the following inequality
\begin{equation}  \label{mrei-ii}
\left( \sum_{P \in \mathcal{P}} \| \prod_{i=1}^k \calE_i f_i \|_{L^1(P \cap Q_R)}^p \right)^\frac1p \leq A(R) C(\bm{\nu}) \prod_{i=1}^k \| f_{i} \|_{L^2(U_i)},
\end{equation}
holds true for any $Q_R$, with $R \ges 1$ and $Q \in \mathcal{P}$. In the above, we tacitly ignore the margin concept, which is necessary since the induction on scales type argument involves wave packets whose use changes slightly the Fourier support of the functions involved at each step. We will run the argument without paying attention to the margin concept; as as explained in Section 5.7 in \cite{Be-6}, the margin can be inserted back into the arguments without issue.  

For a Lebesgue measurable set $A \subset \R^n$ we define the submeasure 
\begin{equation} \label{mudef}
\mu(A)= \sum_{P \subset \mathcal{P}}  \|\prod_{i=1}^k \calE_i f_i (x) \|_{L^\frac2{k-1}(P \cap A)}^p dx.
\end{equation}
It is easy to see that $\mu$ is an submeasure for $p \leq 1$; when $p<1$, $\mu$ fails the additivity property required for measures, but instead it has the sub-additivity property $\mu(A \cap B) \leq \mu (A) + \mu(B)$ if $A \cap B= \emptyset$ which makes it a submeasure. 

When we work within a specific parallelepiped $Q_R$, for each wave $\calE_i f_i$ we employ the wave packet decomposition adapted to $Q_R$; to keep notation simple, this will be implicit in the exposition and we simple use the notation $\calT_i$ for the set of wave packets used. 
If $C \subset Q_R$ is a Lebesgue measurable set, we define
\[
f_{i,C}= \sum_{T \in \calT_i: T \cap C \neq \emptyset} f_{i,T}.
\]
It is an easy exercise to prove the following
\begin{equation} \label{Cloc}
\left( \sum_{P \subset \mathcal{P}} \| \prod_{i=1}^k \calE_i f_i \|_{L^\frac2{k-1}(P \cap C)}^p \right)^\frac1p \leq  C(\bm{\nu}) \left( A(R) \prod_{i=1}^k \| f_{i,C} \|_{L^2(U_i)} + O(R^{-N}) \prod_{i=1}^k \|f_i\|_{L^2(U_i)} \right). 
\end{equation}

For clarity, we first summarize all the key steps of the argument and highlight the adjustments needed for our more general set-up compared to \cite{Be-6}.

We start by applying the polynomial partitioning lemma which allows us to separate into cases based on where the sub-measure concentrates its mass. If the sub-measure concentrates in the cells, away from $Z(P)$, then the argument closes fairly easily using the induction hypothesis. If the sub-measure concentrates near $Z(P)$ and the main contributions come from the non-tangential tubes then using the induction hypothesis we can again obtain good estimates. 

The most interesting case is when the main contributions involve the tangential tubes of at least one family. The idea here is to leverage the extra geometric structure of the tangential tubes to extract localization properties for the new function generated by the family of tangential wave packets. If we manage to do so then we can apply Theorem \ref{MB} to obtain an improved $L^{\frac{2}{k-1}}$ estimate. We can then interpolate this with a simpler $L^{\frac{2}{k}}$ estimate to obtain the desired $L^p$ estimate.

In order to extract localization for the function generated by the tangential wave packets we need at least one of the other families of tubes to ``cooperate". Practically this means that at least one family should be non-transversal to $Z(P)$ in the region of interaction. In order to define this notion of transversality we need the tangent plane $T_z Z$ to be stable over the relevant regions. This is doable except for possibly some portions of $Z$ which can be shown to belong to the neighborhood of a subvariety $Z(P,Q)\subset Z(P)$ instead. For the purpose of this exposition, let us assume that reducing the dimension of the variety makes things easier and therefore we can assume that we only work with the regular portions of $Z(P)$. In particular, we indeed have a well-defined notion of transversal tubes to the variety.

If the main contributions from one family come from transversal tubes then we can re-apply the polynomial partitioning lemma inside the ``wall" by essentially looking at the projection along this transversal direction. The cellular case here closes by the same argument as before while the algebraic case corresponds to further reducing the analysis to a subvariety $Z(P,Q)\subset Z(P)$ which, again, morally makes things easier. Thus, we can assume that all contributions come from non-transversal tubes which ultimately allows us to extract the desired localization properties for the tangential tubes and close the entire argument.

To rigorously implement the dimension reduction arguments we need to run an extra induction on the dimension of the variety which we are studying. This leads to an additional technical layer to the argument especially because it is easier to work directly with the general set-up where the sub-measure is concentrated in the neighborhood of some algebraic variety $Z=Z(P_1,...P_m)$, which then necessitates tracking all the corresponding parameters for each polynomial throughout the argument. Note that the base case of the induction here corresponds to the case $m=n$, which is fairly straightforward as $Z$ is simply a finite set of points in this case.

We now proceed with the actual proof of Proposition \ref{indp}. For some of the technical details which are identical to the ones in \cite{Be-6}, we will simply provide the appropriate reference.

\subsection{Polynomial partitioning - part I} 

We apply Theorem \ref{PP-om} to $\mu$ restricted to $Q_R$, that is to $\mu^{Q_R}(A)=\mu(A \cap Q_R)$; this produces a polynomial $P$ of degree at most $D_1$ and with all the properties listed there. Here the important properties are
$\R^n \setminus Z(P)=\bigcup_j O_j$, the cardinality of the set of $j$'s is $\approx D_1^n$ and $\mu(O_j) \ges D^{-n}_1 \mu (Q_R)$. 

We begin our analysis with the non-algebraic case where the measure is concentrated away from the algebraic variety; thus we assume that the following holds true:
\begin{equation} \label{nalg1}
\mu(N_{R^{\frac12+ \delta_0}}(Z) \cap Q_R) \ll D_1^{-n} \mu(Q_R).
\end{equation}
Here $\delta_0$
is a small parameter with the property $ \delta \ll \delta_0 \ll 1$. 
If we let $O_j'=O_j \setminus N_{R^{\frac12+ \delta_0}}(Z) $, then the following holds true
\[
\mu(O_j') \ges D_1^{-n} \mu(Q_R),
\]
for all $j$'s. For each $f_i$ and each cell $O_j'$, we collect the relevant tubes from $f_i$ that go through $O_j'$ by defining
\[
f_{i,j} = \sum_{T \in \calT_i: T \cap O_j' \neq \emptyset} f_{T}. 
\]

The key gain in the cellular case is that for every $i$ and for any tube $T \in \calT_i$ of radius $R^{\frac12+\delta}$, $T$ enters at most $D_1+1$ many cells $O_j'$. From \cite[Lemma 3.2]{Gu-I} we can then estimate for each $i$
\begin{equation} \label{encount}
\sum_j \| f_{i,j} \|^2_{L^2} \lesssim D_1 \| f_{i} \|^2_{L^2}.
\end{equation}
Since there are $\approx D_1^{n}$ cells $O_j'$ it follows that for most cells we have $ \| f_{i,j} \|_{L^2}^2 \lesssim D_1^{1-n} \| f_{i} \|^2_{L^2}$.  Assuming $D_1$ is large enough, it follows that for most cells the above inequality holds true for every $i$ as well. We note here that $D_1$ and later choices of degrees will be absolute constants independent of $R$. By choosing such a cell $O_j'$ the following holds true
\[
\begin{split}
 \mu(Q_R) \les D_1^{n} \mu(O_j') & \lesssim D_1^{n}   C(\bm{\nu})^p \left( A(R)  \prod_{i=1}^k \| f_{i,j} \|_{L^2} + O(R^{-N}) \prod_{i=1}^k \| f_{i} \|_{L^2(U_i)}\right)^p \\
 & \lesssim C(\bm{\nu})^p \left( A(R)^p D_1^{n} D_1^{\frac{(1-n)kp}2} + O(R^{-Np}) \right)  \prod_{i=1}^k \| f_{i} \|_{L^2}^p. 
 \end{split}
\]
If the power of $D_1$ is negative then we can choose $D_1$ large enough so that $C D_1^{n+\frac{(1-n)kp}2} < \frac12$ (where $C$ is the constant used in the definition of $\les$) and obtain a satisfactory estimate for this case:
\begin{equation} \label{gest1}
  \mu(Q_R)^\frac1p \leq C(\bm{\nu}) (\frac12 A(R) + O(R^{-N})) \prod_{i=1}^k \| f_{i} \|_{L^2}. 
\end{equation}
For this to happen we need $p> \frac{2n}{k(n-1)}$, which is true. 
 
Next, we assume that we are in the opposite scenario to \eqref{nalg1}, that is
\begin{equation} \label{alg1}
\mu(N_{R^{\frac12+ \delta_0}}(Z) \cap Q_R) \gtrsim D_1^{-n} \mu(Q_R),
\end{equation} 
which is referred to as the algebraic case. From the structural properties of Theorem \ref{PP}, it follows that there exists a polynomial $P_1$, occurring in the factorization of $P$, such that $\deg P_1 \leq D_1$, $Z(P_1)$ is a transverse complete intersection and
\begin{equation} \label{alg2}
\mu(N_{R^{\frac12+ \delta_0}}(Z(P_1)) \cap Q_R) \gtrsim \frac{\mu(Q_R)}{D_1^n \ln{D_1}}.
\end{equation}

\subsection{Polynomial partitioning - II} \label{pp2} The analysis of \eqref{alg2} involves in some cases a further reduction to the case when the sub-measure is concentrated on a subvariety of $Z(P_1)$. This is why
 we consider the following more general setup.

We work with a more general algebraic variety $Z=Z(P_1,...,P_{m})$ which is a transverse complete intersection. The polynomials $P_i$'s have degrees less than $D_i$, with $D_1 \leq D_2 \leq .. \leq D_m$. The only natural restriction we impose on $m$ is that $1 \leq m \leq n$. The base case $m=n$ is a bit special and easier and will be treated at the end of the argument in Section \ref{special}. Thus we will assume that $1 \leq m \leq n-1$.

We assume that the following holds true.
\begin{equation} \label{alg4}
\mu(N_{R^{\frac12+ \delta_0}}(Z)  \cap Q_R) \gtrsim \frac{\mu(Q_R)}{D_m^{\kappa_m}},
\end{equation} 
for some $\kappa_m \in \N$.

In order the continue with our arguments we introduce the concept of transversal tubes/waves. Essentially a tube is transversal to $Z$ if it makes a large angle with the tangent plane at $Z$; this is rather imprecise and its exact formalization is made below. 

We seek to locate a relevant portion of $N_{R^{\frac12+\delta_0}} Z$ where the tangent space at $Z$ is close to a specific choice.  We pick $0 < \gamma_0 \ll 1$, whose smallness depends on parameters such as $\nu$, but not $R$ or $\delta$'s. We say that  $B_{R^{\frac12+\delta_0}}(x_0) \subset N_{R^{\frac12+\delta_0}} Z$ is a regular ball if, on each connected component of $Z \cap B_{R^{\frac12+\delta_0}}(x_0)$, the tangent space  $TZ$ is constant up to an angle $\gamma_0$. 

In Section 5.2 of \cite{Be-6} we have established the following dichotomy:

\begin{itemize}
    \item The sub-measure is concentrated on a subvariety $Z(P_1,...,P_m,P_{m+1})$ of $Z$, which is a complete transverse intersection; this case is treated via the induction process with respect to $m$. 

    \item There exists an $(n-m)$-plane $V$ with the property that if we let $\mathcal{B}_V$ be the set of regular balls $B$ satisfying $\angle (V_B,V) \leq \gamma_0$ then 
    \begin{equation} \label{muBR}
    \mu(\bigcup_{B \in \mathcal{B}_V} B ) \gtrsim \frac{\mu(Q_R)}{D_m^{\kappa_m}};
    \end{equation}
    In the above we used the following definition: for a regular ball $B=B_{R^{\frac12+\delta_0}}(x_0)$ we pick a point $z \in B \cap Z$ and define $V_B$ to be the $n- m$ tangent plane $T_z Z$. 
\end{itemize}
We remark that \eqref{muBR} was made in \cite{Be-6} for $B_R$ instead of $Q_R$ (on the right-hand side), but the argument for $Q_R$ is identical.  We continue the analysis in the second case where, with $N_{m}=(\bigcup_{B \in \mathcal{B}_{V_m}} B) $ and $\mu_{m}$ being the restriction of $\mu$ to $N_{m}$, we have $\mu_{m}(N_{m}) \gtrsim \frac{\mu(Q_R)}{D_m^{\kappa_m}}$. The advantage we gain is that within $N_m$ we can introduce the concept of transversal waves. Let $V_{\geq \gamma}=\{ v \in \R^n: \angle (v,V) > \gamma\}$ and 
$V_{\leq  \gamma}=\{ v \in \R^n: \angle (v,V) \leq \gamma\}$. 
 We split each $f_i$ as follows
\[
f_i = f_{i,ntrans} + f_{i,trans},
\]
where
\[
f_{i,trans}=\sum_{T \in \calT_i: v(T) \in V_{> 4 \gamma_0}} f_{i,T}, \qquad f_{i,ntrans}= \sum_{T \in \calT_i: v(T) \in V_{\leq  4 \gamma_0}} f_{i,T}. 
\]
Then we have 
\[
\mu_m(N_m)^p \lesssim  \sum_{car} \sum_{P \subset \mathcal{P}} \| \prod_{i=1}^k \calE_i f_{i,car_i} \|_{L^\frac{2}{k-1}( P \cap N_{m})}^p 
\]
where $\sum_{car}$ runs over all possible combinations $car=(car(1),...,car(k)) \in \{ntrans,trans\}^k$. 

One of these combinations on the right-hand side is dominant, that is there is a choice $car \in \{ntrans,trans\}^k$ such that
\begin{equation} \label{car1}
 \mu_m(N_m)^p \lesssim  \sum_{P \subset \mathcal{P}} \| \prod_{i=1}^k \calE_i f_{i,car_i} \|_{L^\frac{2}{k-1}( P \cap N_{m})}^p, 
\end{equation}
and it suffices to estimate this combination. 

The next dichotomy is based on whether we have $car_i=trans$ for at least $m$ waves or not. 

If $car_i=trans$ for at least $m$ waves, then we can apply the polynomial partitioning and either be able to use induction on scales (like in the cellular case) to obtain
\begin{equation} \label{gest2}
 \mu(Q_R)^\frac1p \leq C(\bm{\nu}) ( \frac12 A(R) + O(R^{-N})) \prod_{j=1}^k \| f_{j} \|_{L^2},
\end{equation}
or further reduce the setup to where the sub-measure is concentrated in the neighborhood $N_{R^{\frac12+\delta}} Z$ where $Z=Z(P_1,...,P_{m}, P_{m+1})$ is a complete transverse intersection. 

The setup of the polynomial partitioning, the derivation of \eqref{gest2} in the cellular case or the reduction to the case when the sub-measure is concentrated in the neighborhood of a subvariety, can be found in Section 5.2 of \cite{Be-6}; no modification is necessary. 

To give some intuition of why things work well if at least $m$ waves are transversal, the idea is that we start with the orthogonal projection $\pi: \R^n \rightarrow V$ and apply the polynomial partitioning for the push-forward sub-measure $\pi_* \mu_m$. This is done using a degree $D_{m+1}$ much larger $D_1,...,D_m$ 
(those appearing in $P_1,...,P_m$) which, within the body of the argument, improves the equivalent of \eqref{encount} for the transversal waves as follows
\[
\sum_j \| f_{i,j} \|^2_{L^2} \lesssim \| f_{i} \|^2_{L^2},
\]
instead of having a factor of $D_{m+1}$; here $\les$ is allowed to depend to $D_1,...,D_m$ but not $D_{m+1}$. 
This is key in making the cellular case work precisely at the threshold when there are at least $m$ transversal waves.

\subsection{The algebraic case - second take} \label{NCC}  Here we consider the case when on the right-hand side of $\eqref{car1}$ there are at most $m-1$ waves which are transversal.  In fact, by choosing $\gamma_0$ small enough with respect to the parameter $\nu$ in \eqref{trans}, it follows that there are exactly $m-1$ waves that are transversal and $n-m$ that are non-transversal.

We recall here that we work with $m \leq n-1$, thus we have at least one wave whose main contribution comes from non-transversal wavepackets; otherwise the arguments below would be vacuous. To keep the notation simpler, we drop the subscripts notation $trans, ntrans$ from the functions involved and simply keep in mind that the first $m-1$ waves are transversal and the last $n-m$ are not. 

We decompose $Q_R$ into parallelepipeds $Q^l$ of size $\rho$, that is they are translates of $Q_\rho$ with 
\[
\rho^{\frac12+\delta_{1}}=R^{\frac12+\delta_0},
\]
where $\delta_1$ is a parameter chosen such that $\delta_0 \ll \delta_1\ll 1$. 
For each $l$ and $i \in \{1,...,k\}$ we select the tubes from $\calT_i$ that intersect $Q^l \cap N_{R^{\frac12+\delta_0}}(Z) \cap N_m$, that is
\[
\calT_{i,l}= \{ T \in \calT_i \mid T \cap N_{R^{\frac12+\delta_0}}(Z) \cap N_m \cap Q^l \ne \emptyset \},
\]
and let $f_{i,l}= f_{i, \calT_{i,l}}$, which we recall is defined as
\[
f_{i,\calT_{i,l}}= \sum_{T \in \calT_{i,l}} f_{i,T}.
\]
The following holds true:
\[
\| \prod_{i=1}^k \calE_i f_i \|_{L^p(Q_R \cap N_m)}^p  \les \sum_l \| \prod_{i=1}^k \calE_i f_{i,l} \|_{L^p(Q^l \cap N_m)}^p + O(R^{-N}) \prod_{i=1}^k \| f_i \|_{L^2}^p, 
\]
since the tubes that do not intersect $N_m$ have their contribution estimated by $O(R^{-N})$.  

The tubes in each family $\mathcal{T}_{i,l}$ are then further classified as tangential or non-tangential - note that these definitions are specific to each $\mathcal{T}_{i,l}$. These definitions are essentially the same as those in \cite{Gu-I, Gu-II}. The only difference is the usage of the curvature parameters $\nu_i$ and parallelepipeds instead of balls, which balance out to have no effect on the geometric implications of these definitions.

A tube $T \in \mathcal{T}_{i,l}$ is called tangential to $Z$ in $Q^l$ if for any $x \in T$ and $y \in Z \cap 2Q^l$ with 
$|x-y| \leq 2 \rho^{\frac12+\delta_1}$, we have $\angle(T,T_y Z) \leq \nu_i \rho^{-\frac12+\delta_1}$. This implies that
\begin{align*}
    T \cap Q^l \subset N_{2 \rho^{\frac12+\delta_1}} Z \cap Q^l.
\end{align*}

A tube $T \in \mathcal{T}_i$ is called non-tangential to $Z$ in $Q^l$ if it is not tangential, that is there exists $x \in T$ and $y \in Z \cap 2B_l$ with $|x-y| \leq 2 \rho^{\frac12+\delta_1}$ such that $\angle(T,T_y Z) > \nu_i \rho^{-\frac12+\delta_1}$. 
This implies that
\[
N_{2 \rho^{\frac12+\delta_1}} (T) \cap 2Q^l \cap Z_{v(T), > \nu_i \rho^{-\frac12+\delta_1}} \neq \emptyset. 
\]

The set of tangential tubes to $Z$ in $Q^l$ is denoted by $\calT_{i,l,ta}$ and the complementary set in $\calT_{i,l}$ is denoted by $\calT_{i,l,nta}$. 
Using these families of wavepackets, we then define
\[
f_{i,l,ta}= f_{i,\calT_{i,l,ta}}, \qquad f_{i,l,nta}= f_{i,\calT_{i,l,nta}}.
\]
If $R$ is large enough, a tube that is tangent to $Z$ in $Q^l$ cannot be transversal and vice-versa. Thus the first $m-1$ waves contain only tubes that are non-tangential (or with negligible contribution in $Q^l$). The other waves contain only tubes that are non-transversal, and they can be either tangential or non-tangential to $Z$ in $Q^l$. A priori the largest angle that a non-tangential tube can make with $Z$ can take any value, but in the current context things are a bit more rigid. Precisely, within $N_m$, the angle of such a tube with $Z$ is bounded from above by $2 \gamma_0$.

Then we continue with
\[
\sum_l \sum_{P \in \mathcal{P}} \| \prod_{i=1}^k \calE_i f_{i,l} \|_{L^p(P \cap Q^l \cap N_m)}^p \lesssim \sum_l  \sum_{car} \sum_{P \in \mathcal{P}} \| \prod_{i=1}^k \calE_i f_{i,l,car(i)} \|_{L^p(P \cap Q^l \cap N_m)}^p. 
\]
where $\sum_{car}$ runs over all possible combinations $car=(car(1),...,car(k)) \in \{nta \}^{m-1} \times \{ta,nta\}^{n-m}$. One of these combinations on the right-hand side is dominant, that is there is a choice $car \in \{nta \}^{m-1} \times \{ta,nta\}^{n-m}$ such that
\[
\sum_l \sum_{P \in \mathcal{P}} \| \prod_{i=1}^k \calE_i f_{i,l} \|_{L^p(Q^l \cap N_m)}^p \les  \sum_l \| \prod_{i=1}^l \calE_i f_{i,l,car(i)} \|_{L^p(Q^l \cap N_m)}^p, 
\]
and it suffices to estimate this combination. We encounter two scenarios: $car(i)=nta$ for all $i=1,...,k$ or $car(i)=ta$ for at least one $i \in \{1,...,k\}$. 
The following two subsection will address these two scenarios. 

\subsection{Non-tangential tubes dominate} \label{nttd} All tubes are non-tangential, that is $car \in \{nta\}^k$. From Lemma \ref{Zan} it follows that a tube $T \in \calT_i$ can be non-tangential to $Z$ in at most $D_m^n$ balls of radius $\nu_i^{-1} \rho^{\frac12}$; practically this translates into the fact that $T \in \calT_i$ can be non-tangential to $Z$ in at most $D_m^n$
parallelepipeds $Q^l$. Therefore if $car(i) \in \{nta\}$, the following holds true:
\begin{equation} \label{trineq}
\sum_l \| f_{i,l,car(i)} \|_{L^2}^2 \les D^n_m \| f_i \|_{L^2}.
\end{equation}
The induction hypothesis on each ball $Q^l$ gives us:
\[
\begin{split}
& \left( \sum_{P \in \mathcal{P}} \| \prod_{i=1}^k \calE_i f_{i,l, car(i) } \|^p_{L^\frac2{k-1}(P \cap Q^l \cap N_m) }\right)^\frac1p \\
\leq & C(\bm{\nu}) A(\rho) \prod_{i=1}^k \| f_{i,l,car(i)} \|_{L^2(U_i)} + C(\bm{\nu}) O(R^{-N}) \prod_{i=1}^k \| f_{i} \|_{L^2(U_i)}.  
\end{split}
\]
We sum with respect to $l$ to obtain:
\[
\begin{split}
& \left( \sum_l \sum_{P \in \mathcal{P}} \| \prod_{i=1}^k \calE_i f_{i,l, car(i)} \|^p_{L^\frac2{k-1}(P \cap Q^l \cap N_m)} \right)^\frac1p \\
\leq &  C(\bm{\nu}) A(\rho) \left(  \sum_l \prod_{i=1}^k \| f_{i,l,car(i)} \|_{L^2(U_i)}^p \right)^\frac1p + C(\bm{\nu})O(R^{-N}) \prod_{i=1}^k \| f_{i} \|_{L^2(U_i)}  \\
\lesssim &  C(\bm{\nu}) A(\rho) D_m^{\frac{nk}2}  \prod_{i=1}^k \| f_{i} \|_{L^2(U_i)} + C(\bm{\nu}) O(R^{-N}) \prod_{i=1}^k \| f_{i} \|_{L^2(U_i)},
\end{split}
\]
where, in passing to the last line, we have used \eqref{trineq} to obtain
\[
\begin{split}
\left( \sum_l \prod_{i=1}^k \| f_{i,l,car(i)} \|_{L^2(U_i)}^{\frac2k} \right)^{\frac{k}2} \lesssim \prod_{i=1}^k  \left( \sum_l \| f_{i,l,car(i)} \|_{L^2(U_i)}^2 \right)^{\frac12} \lesssim D_m^{\frac{nk}2} \prod_{i=1}^k \| f_{i} \|_{L^2(U_i)},
\end{split}
\]
which then is interpolated ($\frac2k < p < \infty$) with the trivial inequality
\[
\max_l \prod_{i=1}^k \| f_{i,l,car(i)} \|_{L^2(U_i)} \les \prod_{i=1}^k \| f_{i} \|_{L^2(U_i)}. 
\]
Therefore in this case we conclude with the following inequality:
\begin{equation} \label{gest3}
\mu(Q_R)^\frac1p \lesssim C(\bm{\nu}) \left( D_m^{\kappa_m + \frac{nk}2} A(\rho) + O(R^{-N}) \right) \prod_{i=1}^k \| f_{i} \|_{L^2(U_i)}. 
\end{equation}

\subsection{At least one family of tubes is tangent} Here we estimate the following term
\[
\sum_l \sum_{P \in \mathcal{P}} \| \prod_{i=1}^k \calE_i f_{i,l,car(i)} \|^p_{L^\frac2{k-1}(P \cap Q^l \cap N_m)}
\]
where $car(m)= tang$; this choice does not restrict the generality of the argument. We recall that the first $m-1$ waves are transversal, hence non-tangential, the $m$-th wave is tangential, while the other waves are non-transversal but each can be either tangential or non-tangential.

Consider a parallelepiped $Q_{R^{\frac12+\delta_0}}(x_0) \subset N_m$; here $x_0$ is meant to be the center of the parallelepiped. With $V$ being the $n-m$ plane that has been used in the definition of $N_m$, we recall the following geometric information that we have so far for the current setup:

\begin{enumerate}[label=(\roman*)]
    \item For any $i \in \{1,...,m-1\}$ and for any $T \in \calT_i$, $\angle(v(T),V) > 4 \gamma_0$;

    \item For any $i \in \{m+1,...,n-1\}$ and for any $T \in \calT_i$, $\angle(v(T),V) \leq 4 \gamma_0$;

    \item For any $T \in \calT_{m}$ with $T \cap Q_{R^{\frac12+\delta_0}}(x_0) \neq \emptyset$, $\angle(v(T), T_{z_0} Z) \leq \rho^{-\frac12+\delta_1}$, for some $z_0 \in Q_{R^{\frac12+\delta_0}}(x_0) \cap Z$;

    \item $\angle(V,T_{z_0} Z) \leq \gamma_0$.
\end{enumerate}

We let
\[
f_{m,l,ta,Q_{R^{\frac12+\delta_0}(x_0)}} = \sum_{T \in \calT_{m,l,ta}: T \cap Q_{R^{\frac12+\delta_0}(x_0)} \neq \emptyset} f_{m,T}. 
\]
The above geometric information implies that $\mathcal{F} (\calE_m f_{m,l,ta,Q_{R^{\frac12+\delta_0}(x_0)}})$ is supported in the following subset of $S_m$
\[
\tilde S_m = \{ \zeta \in S_m: | \pi_{(T_{z_0} Z)^\perp} N_m(\zeta)| \leq \nu_m \rho^{-\frac12+\delta_1}  \}.
\]
We intend to invoke Lemma \ref{kL}, but some more work is needed before we can do so. To put things in perspective, $(T_{z_0} Z)^\perp$
has dimension $m$, thus hinting at localization properties in $m$ directions; but some of the directions in $(T_{z_0} Z)^\perp$ can be very close
to a linear combination of $N_1(\zeta_1),...,N_{m-1}(\zeta_{m-1})$  making them useless for the purpose of extracting good localization information. However it is clear that there is at least one good direction based on dimensionality considerations. 

We make a fixed choice of $\bar \zeta_i \in S_i, i \in \{1,...,k\} \setminus \{m\}$, but let $\zeta_m$ vary inside $S_m$. \eqref{curva} gives us:
\begin{equation}
| N_1(\bar \zeta_1) \wedge ... \wedge N_{k}(\bar \zeta_{k}) \wedge S_{N_m(\zeta_m)} v | \geq c \nu_m | v|,
\end{equation}
for every $v \in \on{span}(N_1(\bar \zeta_1), ..,  N_{k}(\bar \zeta_{k}))^\perp$. 
Therefore if we let $V_{\zeta_m}= \on{span}(\pi_{T_{\zeta_m} S_m} N_1(\bar \zeta_1), .., \pi_{T_{\zeta_m} S_m} N_{k}(\bar \zeta_k))$, we see that \eqref{Sv} is satisfied. Thus $V_{\zeta_m}^\perp=\on{span}(\pi_{T_{\zeta_m} S_m} N_1(\bar \zeta_1), ..., \pi_{T_{\zeta_{m}} S_{m}} N_{k}(\bar \zeta_{k}))^\perp$, 
where the orthogonal space was implicitly taken inside $T_{\zeta_m} S_m$. 

Next we show that \eqref{ph} is satisfied and for this we define $\calH^\perp=(T_{z_0} Z)^\perp \cap N_1(\bar \zeta_1)^\perp \cap .. \cap N_{m-1}(\bar \zeta_{m-1})^\perp$. $\calH^\perp$ is easily seen to be a one-dimensional subspace: indeed it is orthogonal to $N_1(\bar \zeta_1),...,N_{m-1}(\bar \zeta_{m-1})$ and ``almost" orthogonal to $N_m(\zeta_m),...,N_{k}(\bar \zeta_{k})$ since it is orthogonal to  $T_{z_0} Z$.   We are going to weaken our localization property in that 
$\mathcal{F} (\calE_m f_{m,l,ta,Q_{R^{\frac12+\delta_0}}(x_0)})$ is supported in the following subset of $S_m$
\[
\tilde S_m=  \{ \zeta \in S_m: | \pi_{\calH^\perp} N_m(\zeta)| \leq \nu_m \rho^{-\frac12+\delta_1}  \}.
\]
Let $e \in \calH^\perp$ be the unit vector. Since $\calH^\perp$ is one-dimensional, it suffices to establish \eqref{ph} for $e$.  First we have that $|\angle(e,N_m(\zeta_m)) - \frac{\pi}2| \lesssim \rho^{-\frac12+\delta_1} \ll 1$. Second, for each $i \in \{1,...,m-1\}$ we know that $e \perp N_i(\bar \zeta_i)$; from this it follows that  $|\angle (e,N_i(\zeta_i))-\frac{\pi}2| \leq c \ll 1$ for any choice $\zeta_i \in S_i,  i=1,..m-1$ (this is a consequence of the small diameter of each $S_i$, see \eqref{cN}). Third, for each $i \in \{m+1,...,k\}$, it follows that $|\angle (e,N_i(\bar \zeta_i))-\frac{\pi}2| \les \gamma_0$. Thus, by choosing $c$ and $\gamma_0$ small enough, \eqref{ph} is guaranteed to hold true.

At this point we can invoke Lemma \ref{kL} to conclude that there exists $\tilde S_m'$ a submanifold of $S_m$ of co-dimension one such that 
\[
\tilde S_m \subset B_{c_1\rho^{-\frac12+\delta_1}}(\tilde S_m')
\]
for some $c_1 \les 1$. In addition, $|N_{\zeta_m} S_m' \wedge N_1(\zeta_1) \wedge ... \wedge N_{k}(\zeta_{k})| \ges 1$ (where we skip $N_m(\zeta_m)$)
for any choice $\zeta_i \in S_i, i \in \{1,...,k\} \setminus \{ m\}$ and $\zeta_m \in \tilde S_m'$; this is first obtained for $\bar \zeta_1, .., \bar \zeta_k$ and the extended
to all choices $\zeta_i$ by using the smallness of the diameter of each $S_i$. 

We are ready to invoke Theorem \ref{MB}, but first we fix some notation.
We let
\[
f_{i,Q_{R^{\frac12+\delta_0}}(x_0)}= \sum_{T \in \calT_i: T \cap Q_{R^{\frac12+\delta_0}}(x_0) \neq \emptyset} f_{i,T},
\]
and note that we dropped the indices $l$ and $car(i)$. Given the localization just derived, we invoke \eqref{LfT} to obtain the estimate
\[
\begin{split}
\| \prod_{i=1}^{k} \calE_i f_{i,l,car(i)} \|_{L^\frac{2}{k-1}(Q_{R^{\frac12+\delta_0}}(x_0))} & \lesssim_{\epsilon}  C(\bm{\nu}) (R^{\frac12+\delta_0})^\epsilon
(\rho^{-\frac12+\delta_1})^\frac12 \prod_{i=1}^k \| f_{i,Q_{R^{\frac12+\delta_0}}(x_0))} \|_{L^2(U_i)} \\
& + C(\bm{\nu}) O(R^{-N}) \prod_{i=1}^k \| f_i\|_{L^2(U_i)}.
\end{split}
\]
To conclude the main ingredient in this argument, we use the estimate
\[
\| \prod_{i=1}^{k} \calE_i f_{i,l,car(i)} \|_{L^\frac{2}{k-1}(N_m)} \lesssim  C(\bm{\nu}) (R^{\frac12+\delta_0})^\epsilon (\rho^{-\frac12+\delta_1})^\frac12 
\left( \sum_{x_0} \prod_{i=1}^k \| f_{i,Q_{R^{\frac12+\delta_0}}(x_0))} \|_{L^2(U_i)}^{\frac2{k-1}} \right)^{\frac{k-1}2},
\]
where $x_0$ is chosen such that the set $\{ Q_{R^{\frac12+\delta_0}}(x_0)) \}_{x_0}$ has the finite intersection property and it covers $N_m$. 
 We claim the following estimate 
\begin{equation} \label{aux30}
\left( \sum_{x_0} \prod_{i=1}^k \| f_{i,Q_{R^{\frac12+\delta_0}}(x_0)} \|_{L^2(U_i)}^{\frac2{k-1}} \right)^{\frac{k-1}2} \lesssim \prod_{i=1}^k \| f_{i} \|_{L^2(U_i)}.
\end{equation}
We postpone the proof to Subsection \ref{subsection 5.8}. Taking it for granted for now, we conclude with
\[
\| \prod_{i=1}^{n-1} \calE_i f_{i,l,car(i)} \|_{L^\frac{2}{n-2}(N_m)} \lesssim C(\bm{\nu}) \left( (R^{\frac12+\delta_0})^\epsilon (\rho^{-\frac12+\delta_1})^\frac12 + O(R^{-N}) \right) \prod_{i=1}^{n-1} \| f_{i} \|_{L^2(U_i)},
\]
where we recall that $k=n-1$. This is rewritten as
\[
\left( \sum_{P \in \mathcal{P}} \| \prod_{i=1}^{n-1} \calE_i f_{i,l,car(i)} \|_{L^\frac{2}{n-2}(P \cap N_m)}^{\frac2{n-2}} \right)^{\frac{n-2}2}
\lesssim C(\bm{\nu}) \left( (R^{\frac12+\delta_0})^\epsilon (\rho^{-\frac12+\delta_1})^\frac12 + O(R^{-N}) \right) \prod_{i=1}^{n-1} \| f_{i} \|_{L^2(U_i)}.
\]
At the same time we have the global estimate
\begin{equation} \label{aux111}
\left( \sum_{P \in \mathcal{P}} \| \prod_{i=1}^{n-1} \calE_i f_{i,l,car(i)} \|_{L^\frac{2}{n-2}(P \cap Q_R)}^{\frac2{n-1}} \right)^{\frac{n-1}2}
\lesssim  C(\bm{\nu}) R^{\frac{n-1}2} \prod_{i=1}^{n-1} \| f_{i} \|_{L^2(U_i)}.
\end{equation}
We postpone the proof of this estimate to Subsection \ref{subsection 5.8} too.

Then we interpolate the estimates in $\ell^{\frac2{n-2}}$ with $\ell^\frac2{n-1}$ with $\theta=\frac{2(n-1)}{2n-1}$ and $1-\theta=\frac{1}{2n-1}$ respectively,
to obtain an estimate in  $L^{\frac{2(2n-1)}{(n-1)(2n-3)}}$ (the endpoint for $k=n-1$) with the factor
\[
((R^{\frac12+\delta_0})^\epsilon)^{\frac{2(n-1)}{2n-1}} (\rho^{-\frac12+\delta_1})^{\frac12 \cdot \frac{2(n-1)}{2n-1}}(R^{\frac{n-1}2})^{\frac1{2n-1}} \lesssim 
R^{\epsilon + O(\delta_0 +\delta_1^2)}.
\]
This last inequality suggests that a good choice of parameters is
\[
\delta \ll \delta_0 \ll \delta_1 \ll \epsilon. 
\]
This leads to the following inequality:
\begin{equation} \label{gest4}
\mu(Q_R)^\frac1p \leq C(\epsilon) C(\bm{\nu}) D_m^{\kappa_m} R^{\frac32\epsilon}  \prod_{i=1}^k \| f_{i} \|_{L^2(U_i)},
\end{equation}
where $\epsilon$ and $C(\epsilon)$ are the parameters given by Theorem \ref{MB} and \eqref{aux30}. In particular, for $C(\epsilon)$ we pick the maximum of the two used in those results.

\subsection{The special case $m=n$.} \label{special} Here we cover the special case $m=n$.  We are given an algebraic variety $Z=Z(P_1,...,P_{n})$ which is a transverse complete intersection. The polynomials $P_i$'s have degrees less than $D_i$, with $D_1 \leq D_2 \leq .. \leq D_n$ and the following holds true:
\begin{equation} \label{alg4a}
\mu(N_{R^{\frac12+ \delta_0}}(Z)  \cap Q_R) \gtrsim \frac{\mu(Q_R)}{D_n^{\kappa_n}}.
\end{equation} 

From Theorem \ref{ftci} it follows that $N_{R^{\frac12+ \delta_0}}(Z)$ is a finite union of a collection of balls $\{ B_{R^{\frac12+\delta_0}}(x_i) \}_{i=1}^N$ 
where $N \leq \prod_{i=1}^n D_i \leq D_n^n$. In each of this balls we have
\[
\mu(B_{R^{\frac12+\delta_0}}(x_i))^\frac1p \leq C(\bm{\nu}) A(R^{\frac12+\delta_0}) \prod_{i=1}^k \| f_i \|_{L^2(U_i)}; 
\]
this is because each ball is contained in some parallelepiped $Q_{R^{\frac12+\delta_0}}$; Thus
\[
\mu(N_{R^{\frac12+ \delta_0}}(Z))^\frac1p \lesssim N^\frac1p C(\bm{\nu}) A(R^{\frac12+\delta_0}) \prod_{i=1}^k \| f_i \|_{L^2(U_i)}, 
\]
from which we can conclude with
\begin{equation} \label{gest5}
\mu(Q_R)^\frac1p \lesssim D_n^{\frac{\kappa_n+n}p} C(\bm{\nu}) A(R^{\frac12+\delta_0}) \prod_{i=1}^k \| f_i \|_{L^2(U_i)}. 
\end{equation}

\subsection{Conclusion} \label{concl}
In this section we bring together all the bounds we collected and conclude with the desired bound. 
It is obvious that for $R \les 1$ we have $A(R) \les 1$.  

Putting together all the estimates \eqref{gest1}, \eqref{gest2}, \eqref{gest3}, \eqref{gest4} and \eqref{gest5}, we conclude with
\[
A(R) \leq \max( \frac12 A(R) + O(R^{-N}), C A(R^{\frac{\frac12+\delta_0}{\frac12+\delta_1}}) + O(R^{-N}), C A(R^{\frac12+\delta_0}), C \cdot C(\epsilon)  R^{\frac32\epsilon}). 
\]
Here $C=C(D_1,...,D_n,\kappa_1,...,\kappa_n)$ with the observation that the degrees $D_i$ and powers $\kappa_i$ are universal, independent on $R$. 
From this inequality we can conclude with
\[
A(R) \leq \tilde C(\epsilon) R^{2\epsilon},
\]
for $R \gg 1$. This finishes the proof of Proposition \ref{indp}, modulo the claims \eqref{aux30} and \eqref{aux111} which are discussed in the next section.

\subsection{Argument for \eqref{aux30} and \eqref{aux111}}\label{subsection 5.8}

In this section we cover the proofs of \eqref{aux30} and \eqref{aux111} which were used in the main argument of Proposition \ref{indp}. It is helpful to keep in mind that both estimates hold true regardless of any curvature assumptions. 

\begin{proof}[Proof of \eqref{aux30}] The first observation is that \eqref{aux30} follows from the more general estimate
\begin{equation} \label{aux30a}
\left( \sum_{x_0} \prod_{i=1}^k \| f_{i,Q_{R^{\frac12+\delta_0}}(x_0)} \|_{L^2(U_i)}^{\frac2{k-1}} \right)^{\frac{k-1}2} \lesssim \prod_{i=1}^k \| f_{i} \|_{L^2(U_i)},
\end{equation}
where $x_0$ is allowed to run over a larger set restricted only by $Q_{R^{\frac12+\delta_0}}(x_0) \subset Q_R$; of course we keep the condition that the set of parallelepipeds have the finite intersection property. The strategy here is standard, inspired by the work in \cite{BeCaTa}. We bound $\| f_{i,Q_{R^{\frac12+\delta_0}}(x_0)} \|_{L^2(U_i)}$ in terms of norms of wave packets and invoke the end-point multilinear Kakeya to obtain our result. The end-point multilinear Kakeya result we need comes from the work of Carbery and Valdimarsson \cite{CaVa}, which in turn is based on the work of \cite{Gu-main}, and this is because the work in \cite{CaVa} states the end-point estimate at all levels of multilinearity. 

There is only a small, insignificant technicality that we need to address. The above strategy works well if the wave packets are built on the scale of the uncertainty principle but this is not the case for the families of wave packets $\mathcal{T}_i$ which we used so far. This is why we modify the wave packet structure in each family $\mathcal{T}_i$ by simply replacing the spatial localization from $R^\frac{1+\delta}2$ to $C R^\frac12$ for some large $C$ (here we work on the family of wave packet adapted to a parallelepiped $Q_R$) - this is the more classical way to define wave packets; we call this new family $\tilde{\mathcal{T}}_i$. 

Below we use the $\tilde{\mathcal{T}}_i$ adapted to the $Q_{R^{1+2\delta_0}}$, which is the parallelepiped containing $Q_R$; the reason we do so is because the basic unit upon which we perform summation is of size $R^{\frac12+\delta_0}$. Let $c_{i,T}= \|f_{i,T}\|^2_{L^2}$. It is easy to see that 
\begin{equation} \label{aux30b}
\| f_{i,Q_{R^{\frac12+\delta_0}}(x_0)} \|_{L^2(U_i)} \les \left( \sum_{T \in \tilde{\calT}_i} 
\la R^{-\frac12-\delta_0} d(T, Q_{R^{\frac12+\delta_0}}(x_0)) \ra^{-N} c_{i,T} \right)^\frac12,
\end{equation}
 For now, let us ignore the tubes in each class which do not intersect $Q_{R^{\frac12+\delta_0}}(x_0)$, as they come with additional decay.  For the tubes which intersect $Q_{R^{\frac12+\delta_0}}(x_0)$, we have that
\[
\begin{split}
\prod_{i=1}^k \| f_{i,Q_{R^{\frac12+\delta_0}}(x_0)} \|^{\frac2{k-1}}_{L^2(U_i)}  \les R^{-n(\frac12+\delta_0)} \int_{Q_{R^{\frac12+\delta_0}}(x_0)}
\left( \prod_{i=1}^k \sum_{T_i \in \tilde{\calT}_i} c_{i,T_{i}} \chi_{T_i}(x) \right)^\frac1{k-1} dx 
\end{split}
\]
and as a consequence
\[
\begin{split}
\sum_{x_0} \prod_{i=1}^k \| f_{i,Q_{R^{\frac12+\delta_0}}(x_0)} \|^{\frac2{k-1}}_{L^2(U_i)} & \les R^{-n(\frac12+\delta_0)} \int_{Q_{R}}
\left( \prod_{i=1}^k \sum_{T_i \in \tilde{\calT}_i} c_{i,T_{i}} \chi_{T_i}(x) \right)^\frac1{k-1} dx \\
& \les ( \prod_{i=1}^k \sum_{T_i \in \tilde{\calT}_i} c_{i,T_{i}} )^{\frac1{k-1}},
\end{split}
\]
where in passing to the last inequality we have used \cite[Theorem 1]{CaVa} in a rescaled version: the tubes in \cite{CaVa} have diameter $1$, while here they have diameter $R^{\frac12+\delta_0}$. Since $\|f_i\|_{L^2} \approx (\sum_{T_i \in \tilde{\calT}_i} c_{i,T_{i}} )^\frac12$, \eqref{aux30a} follows. 

For the tubes which do not intersect $Q_{R^{\frac12+\delta_0}}(x_0)$, we can use a similar argument combined with the rapid decay from \eqref{aux30b} to complete the proof of \eqref{aux30a}.

\end{proof}

\begin{proof}[Proof of \eqref{aux111}] We first remark that is suffices to establish the (apparently) stronger estimate
\begin{equation} \label{aux111b}
\left( \sum_{P \in \mathcal{P}} \| \prod_{i=1}^{n-1} \calE_i f_{i} \|_{L^\frac{2}{n-2}(P \cap Q_R)}^{\frac2{n-1}} \right)^{\frac{n-1}2}  
\lesssim  R^{\frac{n-1}2}  C(\bm{\nu}) \prod_{i=1}^{n-1} \| f_{i} \|_{L^2(U_i)}.
\end{equation}
We start with the generic estimate
\[
\| \prod_{i=1}^{n-1} \calE_i f_{i} \|_{L^\frac{2}{n-2}(P)} \les  C(\bm{\nu}) \prod_{i=1}^{n-1} \| f_{i} \|_{L^2(U_i)}.
\]
It is standard that this estimate localizes, in that it suffices to consider the relevant masses within the parallelepiped $P$; precisely we have
\begin{equation} \label{aux111a}
\| \prod_{i=1}^{n-1} \calE_i f_{i} \|_{L^\frac{2}{n-2}(P)} \les  C(\bm{\nu}) \prod_{i=1}^{n-1} \left(\sum_{z \in \Z^{n-1}_i} \la z_i \ra^{-N} \| \calE_i f_{i} \|_{L^\infty_{x_i} L^2_{\bar x_i}(P+ \nu_i^{-1} z_i)} \right),
\end{equation}
where $\Z^{n-1}_i= \{ z \in \Z^n: z_i=0 \}$. For fixed $i$, we have the estimate
\[
\sum_{P \in \mathcal{P}: c_i(P)=constant} \left(\sum_{z \in \Z^{n-1}_i} \la z_i \ra^{-N} \| \calE_i f_{i} \|_{L^\infty_{x_i} L^2_{\bar x_i}(P+ \nu_i^{-1} z_i)} \right)^2 \les \| \calE_i f_{i} \|_{L^\infty_{x_i} L^2_{\bar x_i}}^2 \les \|f_{i}\|_{L^2},
\]
where we recall that $c(P)$ is the center of the parallelepiped $P$. The above is a consequence of the fact that the mass of $\calE_i f_{i}$ within parallelepipeds $P$ with $c_i(P)=constant$ has good orthogonality properties, which in turn is a consequence of the dispersive properties of the operator $\calE_i$. From this last estimate we obtain
\[
\begin{split}
& \sum_{P \in \mathcal{P}: P \subset Q_R} \left(\sum_{z \in \Z^{n-1}_i} \la z_i \ra^{-N} \| \calE_i f_{i} \|_{L^\infty_{x_i} L^2_{\bar x_i}(P+ \nu_i^{-1} z_i)} \right)^2 \\
\les & R^\frac12 \sum_{P \in \mathcal{P}: c_i(P)=constant} \left(\sum_{z \in \Z^{n-1}_i} \la z_i \ra^{-N} \| \calE_i f_{i} \|_{L^\infty_{x_i} L^2_{\bar x_i}(P+ \nu_i^{-1} z_i)} \right)^2 \\
\les &  R^\frac12 \|f_{i}\|_{L^2}. 
\end{split}
\]
From this and \eqref{aux111a} we obtain \eqref{aux111b}.

\end{proof}

\section{Proof of Lemma \ref{lem epsilon removal}}\label{section 6}

The proof of the epsilon removal lemma uses the ideas developed by Tao in \cite{Tao-BR}. Technically, our argument follows the adaptation of those ideas in \cite{Be5}, with additional adjustments to take into account the geometry of the parallelepipeds.

An important concept in the $\epsilon$-removal argument is that of sparseness, which is usually defined for isotropic objects, such as cubes or spheres. In our context we need an anisotropic notion of sparseness for parallelepipeds.

\begin{defin}
Let $\mathcal{Q}=\{Q_k\}_{k=1}^N$ be a collection of parallelepipeds in $\mathcal{P}_R$. Given $C\geq 1$, we say that $\mathcal{Q}$ is $C$-sparse if for any distinct $Q_1, Q_2\in\mathcal{Q}$ we have:
\begin{align*}
    \abs{\nu_i(c_i(Q_1) - c_i(Q_2))}\geq R^C N^C
\end{align*}
for at least one $i$.
\end{defin}

We fix a constant $C$, which will be determined at the end of this section. We then have the following sparseness lemma which follows by simply (anisotropically) rescaling the corresponding result from \cite[Lemma 3.3]{Tao-BR}:

\begin{lemma}\label{lsc}
Let $\mathcal{Q}\subseteq \mathcal{P}$ be a collection of parallelepipeds in $\mathcal{P}$ and let $N\geq 1$. Then there exists $O(N\#\mathcal{Q}^{\frac{1}{N}})$ many sparse collections $\mathcal{O}_l\subseteq \mathcal{P}_{R_l}$ which cover $\mathcal{Q}$ such that $1\lesssim R_l\lesssim \#\mathcal{Q}^{C^N}$.
\end{lemma}

With the above notation and result in place we can proceed with the proof.

\begin{proof}[Proof of Lemma \ref{lem epsilon removal}]
We prove a weak-type estimate using the same strategy from \cite{Be5}. First, we normalize $\norm{f_j}_{L^2(d\sigma_j)}=1$ for all $j$ and for any $\lambda>0$ we define:
\begin{align}
    E(\lambda)\defeq \{P\in\mathcal{P}\mid \|\prod_{i=1}^k \calE_if_i\|_{L^{\frac{2}{k-1}}(P)}\geq\lambda\}
\end{align}
By the generic $k$-linear restriction estimate we have $\norm{\prod_{i=1}^k \calE_if_i}_{L^{\frac{2}{k-1}}(P)}\lesssim C(\bm{\nu})$ and therefore we may assume that $\lambda\lesssim C(\bm{\nu})$.

We apply Lemma \ref{lsc} to $E(\lambda)$ to obtain sparse collections $\mathcal{O}_l=\{Q_j: Q_j \in \mathcal{P}_{R_l}\}_{j=1}^N$. We fix one such sparse collection 
$\mathcal{O}_l$ and seek to provide an estimate on it. For each $Q$ in this collection collection, we invoke the refined estimate 
\begin{equation} \label{aux111aa}
\| \prod_{i=1}^{n-1} \calE_i f_{i} \|_{L^\frac{2}{n-2}(Q)} \les  C(\bm{\nu}) R_l^\epsilon \prod_{i=1}^{n-1} \left(\sum_{z \in \Z^{n-1}_i} \la z_i \ra^{-\bar N} \| \calE_i f_{i} \|_{L^\infty_{x_i} L^2_{\bar x_i}(Q+ \nu_i^{-1} R_l z_i)} \right),
\end{equation}
where $\Z^{n-1}_i= \{ z \in \Z^n: z_i=0 \}$; this is \eqref{aux111a} adapted to the larger parallelepiped $Q$. We also recall the convention that $x=(x_i,\bar x_i)$ where $x_i$ is the coordinate of $x$ in the $e_i$ direction, while $\bar x_i$ collects all the other coordinates. 

Next, we claim that, assuming $\bar N$ is large enough, for each $i$ the following holds true
\begin{equation} \label{ssp}
    \sum_{j=1}^N \sum_{z \in \Z^{n-1}_i} \la z_i \ra^{-\bar N} \| \calE_i f_{i} \|^2_{L^\infty_{x_i} L^2_{\bar x_i}(Q_j+ \nu_i^{-1} R_l z_i)}  \les \|f_i\|^2_{L^2}.
\end{equation}
We postpone the proof for now and continue with our argument. Using \eqref{aux111aa}, Hölder's inequality, and  \eqref{ssp}, for any $p\geq\frac{2}{k}$ we have
\begin{align*}
    \sum_{j=1}^N \sum_{ j=1}^N \|\prod_{i=1}^k \calE_if_i\|_{L^\frac{2}{k-1}( Q_j)}^p &\lesssim R_l^{\e p} C(\bm{\nu})^p \sum_{j=1}^N \prod_{i=1}^k \left(\sum_{z \in \Z^{n-1}_i} \la z_i \ra^{-2 \bar N} \| \calE_i f_{i} \|_{L^\infty_{x_i} L^2_{\bar x_i}( Q_j+ \nu_i^{-1} R_l z_i)} \right)^p\\
    &\leq R_l^{\e p} C(\bm{\nu})^p \prod_{i=1}^k \left( \sum_{ j=1}^N \sum_{z \in \Z^{n-1}_i} \la z_i \ra^{- \bar N} \| \calE_i f_{i} \|^2_{L^\infty_{x_i} L^2_{\bar x_i}(Q_j+ \nu_i^{-1} R_l z_i)} \right)^{\frac{p}{2}}\\
    &\lesssim R_l^{\e p} C(\bm{\nu})^p \prod_{i=1}^k \norm{f_i}_{L^2(d\sigma_j)}^p = R_l^{\e p} C(\bm{\nu})^p.
\end{align*}
In other words, for sparse families the loss depends only on the side length of a single parallelepiped instead of the (much larger) side length of the parallelepiped containing the union.

We then combine these facts together to estimate $\#E(\lambda)$ as follows:
\begin{align*}
    \#E(\lambda) &\lesssim \lambda^{-p} \sum_{P\in E(\lambda)} \|\prod_{i=1}^k \calE_if_i\|_{L^{\frac{2}{k-1}}(P)}^p
    \lesssim \lambda^{-p} \sum_l \sum_{\substack{P\in\mathcal{P}\\P\cap\cup_{\mathcal{O}_l}\neq\emptyset}} \|\prod_{i=1}^k \calE_if_i\|_{L^{\frac{2}{k-1}}(P)}^p\\
    & \lesssim \lambda^{-p} \sum_l R_l^{\e p} C(\bm{\nu})^p
     \lesssim \lambda^{-p} C(\bm{\nu})^p N \cdot \#E(\lambda)^{\frac{1}{N}} (\#E(\lambda)^{C^N})^{\e p} \\
    & = \lambda^{-p} C(\bm{\nu})^p N \cdot \#E(\lambda)^{\frac{1}{N}+\e p C^N},
\end{align*}
where we used the notation $\cup_{\mathcal{O}_l} = \cup_{Q \in  \mathcal{O}_l} Q$. Let $\gamma=\frac{1}{N}+\e p C^N$ so that, after rearranging this inequality, we have
\begin{align*}
    \#E(\lambda)\lesssim C(\bm{\nu})^{\frac{p}{1-\gamma}} N^{\frac{1}{1-\gamma}} \lambda^{-\frac{p}{1-\gamma}}.
\end{align*}
If $q>\frac{p}{1-\gamma}$ then we estimate
\begin{align*}
    \sum_{P\in\mathcal{P}}\|\prod_{i=1}^k \calE_if_i\|_{L^{\frac{2}{k-1}}(P)}^q \lesssim \sum_{\lambda\lesssim C(\bm{\nu})} \lambda^q \#E(\lambda)\lesssim C(\bm{\nu})^{\frac{p}{1-\gamma}} N^{\frac{1}{1-\gamma}} \sum_{\lambda\lesssim C(\bm{\nu})} \lambda^{q-\frac{p}{1-\gamma}}\lesssim C(\bm{\nu})^q N^{\frac{1}{1-\gamma}}.
\end{align*}

By picking $N$ large enough and then $\e$ small enough, we can run this argument for any $q>p$ which concludes the proof; in particular, tracking the parameters carefully shows that this can be done as long as $q\geq p+(p+n+1)\frac{C}{\log(\frac{1}{\e})}$, see \cite{Be5} for details. 
\end{proof}

We conclude this section with the proof of \eqref{ssp}. Given the fast decay in the variable $z_i$, an easy exercise shows that it suffices to establish the weaker variant
\begin{equation} \label{sspw}
    \sum_{j=1}^N  \| \calE_i f_{i} \|^2_{L^\infty_{x_i} L^2_{\bar x_i}(Q_j)}  \les \|f_i\|^2_{L^2}.
\end{equation}
In fact we further simplify this by establishing
\begin{equation} \label{sspww}
    \sum_{j=1}^N  \| \calE_i f_{i} \|^2_{L^2_{\bar x_i}(Q_j \cap \{x_i=c_i(Q_j)\}}  \les \|f_i\|^2_{L^2},
\end{equation}
where we recall that $c(Q_j)$ is the center of $Q_j$, and $c_i(Q_j)$
is the $i$-th coordinate of this center. It is easy to see that the argument provided below for \eqref{sspww} is not sensitive to the specific choice of the hyperplane $\{x_i=c_i(Q_j)\}$, and any other choice would work equally well, thus implying \eqref{sspw}.

To prove this we need a little preparation. We recall that 
\[
\calE_i f_{i} = \int e^{i(\bar x_i \xi +  x_i \varphi_i(\xi))} f_i(\xi)d\xi= e^{ix_i \varphi_i(D)} g_i(\bar x_i),
\]
where $g_i= \check f_i$; here there is a slight abuse of notation as the first $i$ in the phase $ix_i \varphi_i(D)$ is meant to be the imaginary unit $i$, while all the other $i$'s are meant to be natural number indices.  \eqref{sspww} can be rephrased in terms of the boundedness of the following operator $T: L^2_{\bar x_i}(\R^{n-1})\rightarrow \ell^2 L^2_{\bar x_i}(\R^{n-1})$
\begin{align}
    Tg=(\chi_{Q_j \cap \{x_i=c_i(Q_j)\}} e^{ic_i(Q_j)\varphi_i(D)} g)_{j=1}^N.
\end{align}
We estimate the norm of this operator via a $TT^*$ argument. For this purpose we compute the adjoint $T^*: \ell^2 L^2_{\bar x_i}(\R^{n-1})\rightarrow L^2_{\bar x_i}(\R^{n-1})$
\begin{align}
    T^* F=\sum_{j=1}^N e^{-ic_i(Q_j)\varphi_i(D)} (\chi_{Q_j \cap \{x_i=c_i(Q_j)\}} F_j),
\end{align}
and therefore $TT^*: \ell^2 L^2_{\bar x_i}(\R^{n-1})\rightarrow \ell^2 L^2_{\bar x_i}(\R^{n-1})$ is given by
\begin{align}
    TT^* F=\left(\chi_{Q_k \cap \{x_i=c_i(Q_k)\}} e^{ic_i(Q_k)\varphi_i(D)}\sum_{j=1}^N e^{-ic_i(Q_j)\varphi_i(D)} (\chi_{Q_j \cap \{x_i=c_i(Q_j)\}} F_j) \right)_{k=1}^N.
\end{align}
Our main claim is that there exists some $\alpha >0$ such that the following estimate holds true
\begin{equation} \label{mQQ}
    \|\chi_{Q_k \cap \{x_i=c_i(Q_k)\}} e^{i(c_i(Q_k)-c_i(Q_j))\varphi_i(D)} (\chi_{Q_j \cap \{x_i=c_i(Q_j)\}} F_j) \|_{L^2_{\bar x_i}} \les (R^C N^C)^{-\alpha} \| F_j \|_{L^2_{\bar x_i}},
\end{equation}
for any $j \neq k$. We will see that the value of $\alpha$ is independent of $C$, thus by choosing $C$ large enough, \eqref{mQQ} implies the boundedness of $TT^*$, which in turn provides \eqref{sspw}. 

The rest of this argument is dedicated to establishing \eqref{mQQ}. We write
\[
\begin{split}
& (\chi_{Q_k \cap \{x_i=c_i(Q_k)\}} e^{i(c_i(Q_k)-c_i(Q_j))\varphi_i(D)} (\chi_{Q_j \cap \{x_i=c_i(Q_j)\}} F_j) \\
= & (\chi_{Q_k \cap \{x_i=c_i(Q_k)\}} K_i(c_i(Q_k)-c_i(Q_j), \cdot) \ast (\chi_{Q_j \cap \{x_i=c_i(Q_j)\}} F_j),
\end{split}
\]
where 
\begin{align}\label{fkerf}
     K_i(x_i,\bar x_i)=\int e^{i(\bar x_i \xi + x_i \varphi_i(\xi))} \chi(\xi)d\xi,
\end{align}
where $\chi(\xi)$ reflects the support of the original localization. Also, the convolution is taken with respect to the $\bar x_i$ variable. 
For this kernel we have the non-stationary phase type decay estimate
\begin{align}\label{fkerfdecay}
    | K_i(x_i,\bar x_i)| \les \la \bar x_i \ra^{-M}, \  \text{for} \  
    |\bar x_i|  \gg \nu_i |x_i|.
\end{align}
Assume now that $\max_{l \neq i} \nu_l |c_l(Q_j) - c_l(Q_k)| \gg \nu_i | c_i(Q_j) - c_i(Q_k)|$; in this case we have $|\bar c_i(Q_j) - \bar c_i(Q_k)| \gg \nu_i | c_i(Q_j) - c_i(Q_k)|$ and then using Schur's test 
\[
\sup_{\bar x_i \in Q_j \cap \{x_i=c_i(Q_j)\}} \int_{Q_k \cap \{x_i=c_i(Q_k)\}}  |K_i(c_i(Q_k)-c_i(Q_j),\bar x_i - \bar y_i)| d\bar y_i \les (R^C N^C)^{-1},
\]
which further gives the estimate
\[
\|(\chi_{Q_k \cap \{x_i=c_i(Q_k)\}} K_i(c_i(Q_k)-c_i(Q_j), \cdot) \ast (\chi_{Q_j \cap \{x_i=c_i(Q_j)\}} F_j)\|_{L^2_{\bar x_i}} \les (R^C N^C)^{-1} \|F_j\|_{L^2},
\]
thus establishing \eqref{mQQ} in this case. 

We are left with the remaining case $\max_{l \neq i} \nu_l| c_l(Q_j) - c_l(Q_k)| \les \nu_i | c_i(Q_j) - c_i(Q_k)|$; as a consequence, in this case we have $\nu_i | c_i(Q_j) - c_i(Q_k)| \ges R^C N^C$. Here things are a little more subtle. If we tried to use the same kernel $K_i$, then the decay is of order  $\la \nu_i x_i \ra^{-\frac12}$, as we are guaranteed to have a non-zero principal curvature (though degenerate); using this in the estimate using Schur's test above will not deliver the desired result. This is because of powers of $\nu_i^{-1}$ may accumulate when integrating across $Q_k$ and $Q_j$ in directions other than $x_i$; in the classical $\epsilon$-removal argument $\nu_i \approx 1$ and this problem does not occur.

Instead, we establish \eqref{mQQ} using wave packets for $e^{it\varphi_i (D)} f_i$ adapted to the ensemble $Q_j \cup Q_k$, that is we consider the parallelepiped 
$ Q \subset \bar R \mathcal{P}$, that contains both $Q_j$ and $Q_k$. The separation guarantees that $\bar R \geq R^C N^C$. For convenience we set $c(Q_j)=0$ and note that this does not restrict the generality of the argument.

We can replace $F_j$ by $f_j= \tilde \chi_{Q_j} F_j$, for some smooth approximation of $\chi_{Q_j}$. Then \eqref{mQQ} becomes
\begin{equation} \label{mQQ2}
    \|\chi_{Q_k \cap \{x_i=c_i(Q_k)\}} e^{i c_i(Q_k) \varphi_i(D)} (\chi_{Q_j \cap \{x_i=c_i(Q_j)\}} f_j) \|_{L^2_{\bar x_i}} \les (R^C N^C)^{-\alpha} \| f_j \|_{L^2_{\bar x_i}}.
\end{equation}
The wave packet theory gives the decomposition
\[
e^{i x_i \varphi_i(D)} \chi_{Q_j} f_j = \sum_{T \in \mathcal{T}} e^{i x_i \varphi_i(D)} f_T. 
\]
Then using orthogonality we have
\begin{equation} \label{aux22}
\| (\chi_{Q_k \cap \{x_i=c_i(Q_k)\}} (\sum_{T \in \mathcal{T}} e^{i x_i \varphi_i(D)} f_T) \|_{L^2}^2 \les \sum_{T \in \mathcal{T}} \| (\chi_{Q_k \cap \{x_i=c_i(Q_k)\}}  e^{i c_i(Q_k) \varphi_i(D)} f_T \|_{L^2}^2. 
\end{equation}
The fast decay of $e^{i x_i \varphi_i(D)} f_T$ away from the tube $T$ allows us to focus on the tubes $T$ which intersect $Q_j$ and $Q_k$. Let $\mathcal{H}$ be the plane passing through $c(Q_j)$ and $c(Q_k)$, and containing the directions $\{e_1,...,e_{n-1}\} \setminus \{e_i\}$. Then $\mathcal{H}^\perp$ is a line whose direction is given by the vector $\mathcal{N}$, whose coordinates are $0$, except $\mathcal{N}_i$ and $\mathcal{N}_n$, and here we have $\nu_i |\mathcal{N}_i| \ges \mathcal{N}_n$. As a consequence, if $N(\zeta)$ is the direction of a tube $T$, then we have that $|\pi_{\mathcal{H}^\perp} N(\zeta)| \les (\bar R)^{-\frac12+\delta}$. 
Then we invoke Lemma \ref{kL} to conclude that the Fourier support of the relevant tubes belongs to a neighborhood of size $\approx (\bar R)^{-\frac12+\delta}$ of a  submanifold $\tilde S$ of $S_i$ with the additional property \eqref{NS}; by projecting this onto $U_i \subset \R^{n-1}$, the domain of the parametrization of $S_i$ and where $\hat f_j$ is defined, we obtain that the Fourier support of the relevant tubes belongs to a neighborhood of size $\approx \tilde \mu:= (\bar R)^{-\frac12+\delta}$ of a  submanifold $\tilde M$ of $U_i$ with the additional property that $|e_n \wedge T_\xi \tilde M| \ges 1, \forall \xi \in \tilde M$. 

Now we bring in the last piece of information. Since the starting object is $\chi_{Q_j} F_j$ and $Q_j$ has size $R$ in the $e_n$ direction, this brings a local constancy for $\hat f_j$ at scale $R^{-1}$ in the direction of $e_n$. Assume $A_{l}=\{\xi \in \R^{n-1}: |\xi_n-l| \leq R^{-1} \}$ is a set of strips of size $R^{-1}$ in the direction $e_n$. If $A$ is a strip of size $\mu \les R^{-1}$ in the direction $e_n$, that is 
$A=\{\xi \in \R^{n-1}: |\xi_n-c| \leq \mu \}$ for some $c \in \R$, and $A \subset A_{l_0}$, then the following holds true
\[
\| \chi_{A} \hat f_j \|_{L^2} \les (R \mu)^{\frac12} \sum_{l} \la l-l_0 \ra^{-N} \| \hat f_j\|_{L^2(A_l)}. 
\]
Using this and the properties of the localization, we obtain
\[
\|\chi_{N_{\bar \mu} \tilde M} f_j \|_{L^2} \les (R \tilde \mu)^{\frac12} \|f_j\|_{L^2} \approx (R (R^C N^C)^{-\frac12+\delta})^{\frac12} \|f_j\|_{L^2} 
\les R^{-\frac{C}8} N^{-\frac{C}8} \|f_j\|_{L^2}, 
\]
if we assume that $C \geq 8$ and $\delta < \frac18$. Recalling \eqref{aux22},
this allows us to conclude that
\[
\|\chi_{Q_k \cap \{x_i=c_i(Q_k)\}} e^{i c_i(Q_k) \varphi_i(D)} (\chi_{Q_j \cap \{x_i=c_i(Q_j)\}} f_j) \|_{L^2_{\bar x_i}} \les R^{-\frac{C}8} N^{-\frac{C}8} \|f_j\|_{L^2},
\]
which provides \eqref{mQQ2}.

\bibliographystyle{amsplain} \bibliography{HA-refs}

\end{document}